\documentclass[12pt]{article}
\usepackage[utf8]{inputenc}

\usepackage{preprint-layout}
\usepackage{preprint-notation}
\usepackage{doi}
\usepackage{bm}

\begin{document}

\title{Graded Parametric Cut{FEM} and Cut{IGA} for Elliptic Boundary Value Problems\\ in Domains with Corners}

\author{Tobias Jonsson, Mats G. Larson and Karl Larsson}
\date{}

\newcommand{\us}{u_\mathrm{s}}
\newcommand{\ur}{u_\mathrm{r}}
\newcommand{\omegac}{\omega_c}

\newcommand{\hatD}{\widehat{D}}
\newcommand{\hatpi}{\widehat{\pi}}
\newcommand{\hatrho}{\widehat{\rho}}
\newcommand{\hatR}{\widehat{R}}
\newcommand{\hatP}{\widehat{P}}
\newcommand{\hattheta}{\hat{\theta}}
\newcommand{\hatpartial}{\widehat{\partial}}
\newcommand{\hatmcK}{\widehat{\mcK}}
\newcommand{\hatI}{\widehat{I}}
\newcommand{\hatOmega}{\widehat{\Omega}}
\newcommand{\hatV}{\widehat{V}}
\newcommand{\hatA}{\widehat{A}}
\newcommand{\hatK}{\widehat{K}}
\newcommand{\hatnabla}{\widehat{\nabla}}

\newcommand{\hatmu}{\mu}

\renewcommand{\hatf}{\hat{f}}
\renewcommand{\hatg}{\hat{g}}
\renewcommand{\hatr}{\hat{r}}
\renewcommand{\hats}{\hat{s}}
\renewcommand{\hatu}{\hat{u}}
\renewcommand{\hatv}{\hat{v}}
\renewcommand{\hatw}{\hat{w}}
\renewcommand{\hath}{\hat{h}}
\renewcommand{\hatn}{\hat{n}}
\renewcommand{\hatx}{\hat{x}}
\renewcommand{\haty}{\hat{y}}
\renewcommand{\hatpartial}{\hat{\partial}}

\numberwithin{equation}{section} 

\maketitle
%%%%%%%%%%%%%%%%%%%%%%%%%%%%%%%%%%%%%%%%%%%%%%%%%%%%%%%%%%%%%%%%%%%%%%%
\begin{abstract}
We develop a parametric cut finite element method for elliptic boundary value problems with corner singularities where we have weighted control of higher order derivatives of the solution to a neighborhood of a point at the boundary. Our approach is based on identification of a suitable mapping that grades the mesh towards the singularity. In particular, this mapping may be chosen without identifying the opening angle at the corner. We employ cut finite elements together with Nitsche boundary conditions and stabilization in the vicinity of the boundary. We prove that the method is stable and convergent of optimal order in the energy norm and $L^2$ norm. This is achieved by mapping to the reference domain where we employ a structured mesh.
\end{abstract}

\section{Introduction} 

A classical issue in the use and development of finite element methods (FEMs) are problems where the exact solution may be singular at certain points, due to the presence of nonconvex corners or jumps in data. In such cases standard FEMs perform poorly as the low regularity of the solution causes loss of convergence. By utilizing known information about the singularities, various methods that regain optimal order convergence have been devised. In this contribution we consider problems with nonconvex corners in the context of an unfitted finite element method -- the cut finite element method (CutFEM) -- and propose a CutFEM for which we prove stability and optimal order error estimates.

\paragraph{Cut{FEM} and Cut{IGA}.}
CutFEM is a framework for finite element methods where the physical domain may cut the computational mesh arbitrarily while retaining the optimal approximation and stability properties of standard FEMs, see \cite{BuHa2012,BuClHaLaMa15}.
This framework is applicable also to isogeometric analysis (IGA), which
combines powerful spline approximation spaces on structured computational meshes with high precision geometry descriptions via CAD, see, e.g., \cite{IGA,IGABook,MR3372009}. In general, IGA generates a cut computational mesh as the CAD consists of patchwise parametric mappings and trim curves. It is therefore natural to utilize the mathematically rigorous CutFEM framework also in IGA.
We denote this combination of techniques cut isogeometric analysis (CutIGA), which we have previously explored in \cite{ElfLarLar18a,ElfLarLar18b}.

% relation to previous work
\paragraph{Previous Work.} Over the years considerable efforts have been made to construct special finite element methods that deal with singularities arising from domains with (nonconvex) corners or jumps in data, see the very extensive survey \cite{LiLu2000} and the references therein.
A great deal about these singularities is actually known, especially in the planar domain case, and this information is also typically needed to construct a method that efficiently deals with singularities.
The survey \cite{LiLu2000} classifies methods into the following three categories:
\begin{enumerate}[label=\emph{\Roman*.}]
\item \emph{Methods involving local refinement.} These are based on the assumption that the location of the singular point is known, and that the behavior of the solution when approaching the singular point is
\begin{equation} \label{eq:assumption}
u = \mathcal{O}(r^\alpha) \,, \qquad 0<\alpha < 1 \,, \qquad \text{as $r\to 0$} \,,
\end{equation}
where $r$ is the distance to the singular point. Most methods belong to this category, as does the method in the present work, and it includes standard approaches such as $h/p$-FEM,  and various parametric techniques.
If the location of the singular point is unknown it is possible to instead use an adaptive procedure where the approximation space is iteratively tuned to the problem. Such procedures are typically based on a posteriori error indicators and the tuning can consist of local mesh refinement ($h$-refinement), locally increasing the polynomial order ($p$-refinement), and relocation of the mesh ($r$-refinement), see, e.g., \cite{AS1997,BanRan2003,MR3627181}.

\item \emph{Methods supplementing the approximation space with singular functions.}
Here the assumption is that the leading singular functions in the expansion of the solution when $r\to 0$ are known. These can then be supplemented to the usual approximation space.
Typically, one or two singular functions is sufficient to resolve the singular effects. In this category we also find methods such as XFEM/GFEM, see \cite{BelGraVen2009,FriBel2010}.

\item \emph{Combined methods.}
This final category is based on the assumption that the complete expansion of the solution, i.e., both the singular and analytical parts, is locally known in a singular subdomain $r \leq r_0$.
\end{enumerate}
In the context of IGA there have been some recent contributions dealing with corner singularities. A method belonging to the first category was proposed in \cite{JeOhKaKi2013} where the parametric mappings describing the geometry were modified to grade the approximation space appropriately towards the singularity. In another contribution \cite{OhKiJe2014} a method related to the first and second category was presented in which the approximation space was enriched by certain basis functions constructed using push forward operations with a mapping similar to the one in \cite{JeOhKaKi2013}.

The present work belongs to the first category and is based on a mapping in the form of simple radial scaling, which has the effect of a graded mesh. As the CutFEM uses a geometric description of the domain that is independent of the computational mesh, the mapping only affects the computational mesh, not the domain. The simple expression for the mapping allows us to conveniently reformulate the method in a way that avoids numerical instabilities due to large derivatives of the mapping close to the singular point.
Also, using a mapping which is smooth everywhere (except at the singular point) means that the regularity of the discrete approximation space will not be affected, making the method suitable also for IGA approximation spaces of arbitrarily high regularity.

\paragraph{New Contributions.}

We develop a CutFEM for elliptic problems for domains with corners, including treatment of singularities arising at nonconvex corners.
Our approach is based on  identification of a suitable mapping from a reference domain, where a 
quasi-uniform mesh is used, that grades the mesh towards the singular corner. The mapping 
is bijective but has zero derivative in the corner. We employ weak enforcement of Dirichlet 
conditions, allow cut elements at the boundary and stabilize the method using Ghost 
penalty, see \cite{Burman2010,MaLaLoRo2013a}. We prove that the method is stable by mapping to the 
reference domain and utilizing the additional stability provided by the Ghost penalty.
Furthermore, we can prove optimal order a priori error estimates in the energy and $L^2$ 
norm using a bound on derivatives of order $k$ in the reference domain in terms of a weighted 
norm in the physical domain which is bounded for the singular solution.
The analysis holds for polygonal domains and piecewise smooth boundaries.
Furthermore, the method is not sensitive 
to the choice of grading parameter as long as the grading is strong enough. As a consequence, 
the grading that works for a corner with opening angle close to $2\pi$ works for all angles. 
This observation also indicates that the approach may be extended to three dimensional situations, which we plan on investigating in future work.

\paragraph{Outline.} The remainder of this work is organized as follows: In Section~\ref{section:the-model-problem} we formulate the model problem and describe its regularity properties in terms of weighted norms and then we introduce the grading mapping;
in Section~\ref{section:the-method} we present the parametric cut finite element method, expressed both in the physics domain and in the reference domain;
in Section~\ref{section:error-est} we prove our error estimates;
in Section~\ref{section:numerics} we present numerical examples using tensor product elements of various order and regularity;
and finally, in Section~\ref{section:conclusions} we give some concluding remarks.

\section{The Model Problem and Regularity Properties} \label{section:the-model-problem}

\subsection{Model Problem}
Let $\Omega$ be a domain in $\IR^2$ with boundary $\partial \Omega$ 
and consider the elliptic problem: find $u:\Omega\rightarrow \IR$ such that 
\begin{alignat}{2} \label{eq:dirichlet-problem-a}
-\Delta u  &= f
&\qquad &\text{in $\Omega$}
\\ \label{eq:dirichlet-problem-b}
u &= g
&\qquad & \text{on $\partial \Omega$}
\end{alignat}
We assume that the boundary $\partial \Omega$ consists of a finite number of smooth 
curve segments $\{\Gamma_i\}$ that meet in corners $\mcC = \{ c_j\}$, some of which are 
nonconvex. Such a domain is illustrated in Figure~\ref{fig:schematic-a}. For brevity let us from now on consider the situation that we have one corner $c$ 
which is nonconvex, as illustrated in Figure~\ref{fig:schematic-b}. Since the singular behavior is local, the extension to several singular corners is straightforward.

\begin{figure}
\centering
\begin{subfigure}[t]{0.45\linewidth}\centering
\includegraphics[width=0.40\linewidth]{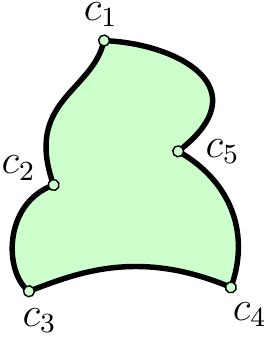}
\subcaption{Domain with multiple corners}
\label{fig:schematic-a}
\end{subfigure}
\begin{subfigure}[t]{0.45\linewidth}\centering
\includegraphics[width=0.5\linewidth]{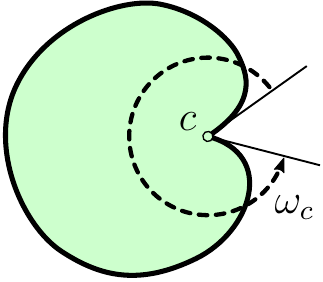}
\subcaption{Domain with a single corner}
\label{fig:schematic-b}
\end{subfigure}
\centering
\caption{Schematic illustrations of domains with corners.
(a) Domain with multiple corners $\{c_j\}$ of which $c_2$ and $c_5$ are nonconvex.
(b) Idealized domain with a single nonconvex corner $c$ with opening angle $\omega_c \in (\pi,2\pi)$.
}
\label{fig:schematic}
\end{figure}

\subsection{The Regularity of the Solution}
The solution $u$ to the above problem with a single nonconvex corner can be decomposed into a regular and singular part
\begin{equation}
u = \ur + \us
\end{equation}
where the regular part fulfills the elliptic shift property
\begin{equation}
\| \ur \|_{H^s(\Omega)} \lesssim \| f \|_{H^{s-2}(\Omega)} + \| g \|_{H^{s-1/2}(\partial\Omega)}
\end{equation}
In the vicinity of the corner $c$ the singular part of the solution takes 
the form
\begin{equation}\label{eq:singular-part}
\us(r,\theta) = r^{\pi/\omega_c} \sin(\theta \pi/\omegac) 
\end{equation}
where the parameter $\omegac \in (\pi,2\pi)$ is the opening 
angle, see Figure~\ref{fig:schematic-b}. Here we use polar coordinates $\{r,\theta\}$ centered at $c$ and with the angular component $\theta$ measured counterclockwise from the tangent line originating in $c$ when passing the boundary in the counterclockwise direction.  Since $1>\pi/\omegac>1/2$ the solution is always in $H^1(\Omega)$ but not in 
$H^2(\Omega)$. More precisely the regularity can be described using weighted norms, see e.g. \cite{MR0226187}.

\begin{rem}[More General Problems]
For brevity we focus on the Dirichlet problem  (\ref{eq:dirichlet-problem-a})--(\ref{eq:dirichlet-problem-b}). In literature more general boundary conditions are studied, for example mixed boundary conditions with both non-homogeneous Neumann and Dirichlet parts that gives singular parts of the solution different from (\ref{eq:singular-part}). Another source than nonconvex corners for inducing singular solutions pertains to jumps in data. However, in most cases the method presented below will still be applicable as the essential assumption is (\ref{eq:assumption}) rather than the complete expression for the singular part of the solution.
\end{rem}

\begin{definition}[Total Derivative Magnitude]\label{def:total-derivative-polar}
The magnitude of the total derivative of order $k$ in polar coordinates is given by
\begin{equation}
|D^k v|^2
=
(\partial_r^k v)^2
+
\sum_{\substack{m+n \leq k, \\ m\geq 0 ,\, n\geq 1}} \left( r^{m-k}\partial_r^m \partial_\theta^n v \right)^2
\end{equation}
\end{definition}

That this definition is reasonable is readily seen by expanding
$(\partial_r + r^{-1}\partial_\theta )^k v$,
evaluating the derivatives of all coefficients of the form $r^{-\ell}$ using the product rule,
and noting that the total derivative magnitude contains all terms in the expansion modulo integer coefficients.

\begin{lem}[Weighted Sobolev Space] The singular part of the solution satisfies, for $k=1,2,\dots$, 
\begin{equation}
\| \us \|^2_{H^{k,\alpha}(\Omega)} 
= 
\sum_{j=1}^k \| r^{\alpha_j} D^j \us \|^2_{\Omega} < \infty
\end{equation}
where the powers $\{\alpha_j\}$ must satisfy the condition
\begin{equation} \label{eq:power-condition}
j-1 - \frac{\pi}{\omegac} < \alpha_j
\end{equation}
\end{lem}
\begin{proof}
Observe that the derivatives appearing in $D^j \us$
%, i.e., the total derivative for the singular part of the solution,
have the bounds
\begin{align} \label{eq:detail-bound-a}
|\partial_r^j \us|
&=
|\partial_r^j r^{(\pi/\omegac)} \sin(\theta \pi/\omegac)| 
\lesssim 
r^{(\pi/\omegac) - j}
\\ \label{eq:detail-bound-b}
|r^{m-j}\partial_r^m \partial_\theta^n \us |
&=
|r^{m-j}\partial_r^m r^{(\pi/\omegac)} \partial_\theta^n \sin(\theta \pi/\omegac) |
\lesssim 
r^{(\pi/\omegac) - j}
\end{align}
where $m,n$ are given in Definition~\ref{def:total-derivative-polar} above.
Let $R$ be the radius of a ball, which contains $\Omega$ and is centered at $c$.
Using \eqref{eq:detail-bound-a}--\eqref{eq:detail-bound-b} we compute the following bound on the weighted norm
\begin{align}
\| r^{\alpha_j} D^j \us \|^2_\Omega
&
\lesssim
\int_0^R r^{2\alpha_j} \left(r^{(\pi/\omegac)-j}\right)^2 r \, dr 
\\
&=\int_0^R r^{2\alpha_j + 2 (\pi/\omegac)  - 2j + 1} \, dr
\\
&= \left[ \frac{r^{2\alpha_j  + 2 (\pi/\omegac) - 2(j-1)}}{  2 \alpha_j  + 2 (\pi/\omegac) - 2(j-1) )} \right]_0^R 
\end{align}
and we note that this bound is finite for 
\begin{equation}
 (j-1) - (\pi/\omegac) <  \alpha_j
\end{equation}
which is precisely the condition \eqref{eq:power-condition}.
\end{proof}

\subsection{Weighted Sobolov Spaces and Parametric Mapping}
We will now construct a mapping from a reference domain to $\Omega$ such that the pullback of the singular part of the solution $\us$ has full regularity in terms of the reference coordinates.

\paragraph{The Parametric Mapping.}
For convenience we here work with polar coordinates, both for the physical coordinates and for the reference coordinates.
Consider the bijective mapping $F_\gamma: \IR^2 \rightarrow \IR^2$ 
defined by  
\begin{equation} \label{eq:mapping}
F_\gamma: \left[ \begin{matrix}
\hat r \\ \hat \theta
\end{matrix}  \right] 
\mapsto  \left[ \begin{matrix}
{\hat r}^\gamma \\ {\hat \theta}
\end{matrix}
\right] = \left[ 
\begin{matrix} 
r \\ \theta 
\end{matrix}
\right] 
\end{equation}
for $\gamma > 0$, and define the domain in reference coordinates
\begin{equation}
\hatOmega = F_\gamma^{-1} (\Omega)
\end{equation}
For $v:\Omega\to\IR$ we define the pullback
\begin{equation}
v(r) = \hat v(\hat r) : \hatOmega \to \IR
\end{equation}
where we for notational brevity in the calculations below only include the radial coordinate as the angular coordinate is unchanged in the mapping \eqref{eq:mapping}.

\begin{thm}[Equivalence of Norms] (1) For $\gamma>0$, it holds
\begin{equation}\label{eq:norm-eq-H1}
\| \nabla u \|_{\Omega} \sim \| \hatnabla \hat u \|_{\hatOmega}
\end{equation}
where the hidden constants depend only on $\gamma$.
\vspace{1ex}
\\ \noindent
(2)  For $k=1,2,\dots$, and $\gamma>0$, it holds
\begin{equation}\label{eq:norm-eq-Hkalpha}
\| \hatD^k \hatu \|_{\hatOmega}^2
\lesssim
\sum_{j=1}^k
\| r^{\alpha_j} D^j u \|_{\Omega}^2
\end{equation}
where $\hatD^k\hatu$ is the total derivative of order $k$ in the reference domain and
\begin{equation} \label{eq:alphaj-alphak}
\alpha_j = (j - k) + \alpha_k
\qquad\text{with}\qquad
\alpha_k = (k-1)\frac{\gamma-1}{\gamma}
\end{equation}

\end{thm}
\begin{proof} {\bf The First Order Case (\ref{eq:norm-eq-H1}).}  Consider the mapping 
\begin{equation}
r = \hatr^\gamma , \qquad \hatr = r^{1/\gamma}
\end{equation}
with derivatives
\begin{equation}
\frac{dr}{d\hatr} = \gamma \hatr^{\gamma - 1} = \gamma r^{(\gamma  -1 )/\gamma},
\qquad
\frac{d\hatr}{d r} = \gamma^{-1} r^{(1-\gamma)/\gamma} = \gamma^{-1} \hatr^{(1-\gamma)}
\end{equation}
and measures 
\begin{equation}
r dr d \theta 
=  \gamma \hatr^{2(\gamma - 1)}  \hatr d\hatr d\hattheta, 
\qquad 
\hatr d \hatr d \hattheta  =  \gamma^{-1} r^{(1-\gamma)/\gamma}  dr d \theta
\end{equation}
Using that $u(r) =\hatu (\hatr )$ we by the chain rule have
\begin{equation}
\frac{du}{dr} = \frac{d\hatu}{d\hatr}\frac{d\hatr}{dr} = \frac{d\hatu}{d\hatr} \gamma^{-1} \hatr^{(1-\gamma)}
\end{equation}
Taking the norm of the gradient and changing coordinates we find the equivalence
\begin{align}
\|\nabla u \|^2_{\Omega} 
&=\int_\Omega \Big(|\partial_r u |^2 + r^{-2} |\partial_\theta u |^2 \Big)   r \, dr d\theta
\\
&=\int_{\hatOmega}
\Big( |\hatpartial_{\hatr} \hatu  |^2  ( \gamma^{-1}  \hatr^{(1-\gamma)} )^2 
+ 
\hatr^{-2\gamma} |\hatpartial_{\hattheta} \hatu |^2 \Big)
\gamma \hatr^{2(\gamma - 1)} 
  \hatr \, d\hatr d\hattheta
\\
&=\int_{\hatOmega}
\Big( \gamma^{-1} |\hatpartial_{\hatr} \hatu  |^2  + \gamma \hatr^{-2} |\hatpartial_{\hattheta} \hatu |^2 \Big) \hatr \, d\hatr d\hattheta
\\
&\sim\int_{\hatOmega}
\Big( |\hatpartial_{\hatr} \hatu  |^2  +  \hatr^{-2} |\hatpartial_{\hattheta} \hatu |^2 \Big) \hatr \, d\hatr d\hattheta
\\
&=\|\hatnabla \hatu \|^2_{\hatOmega}
\end{align}
with constants only dependent on $\gamma$ and $\gamma^{-1}$.

\paragraph{The General Case (\ref{eq:norm-eq-Hkalpha}).}
By Definition~\ref{def:total-derivative-polar} the total derivative magnitude is
\begin{equation} \label{eq:terms-of-that-form}
|\hatD^k \hatv |^2
=
\underbrace{
(\hatpartial_{\hatr}^k \hatv )^2
}_{I}
+
\overbrace{
\sum_{\substack{ m\geq 0 ,\, n\geq 1 \\ k \geq m + n }}
\left(
\hatr^{m-k} \hatpartial_{\hatr}^m \hatpartial_{\hattheta}^n \hatv
\right)^2
}^{II}
\end{equation}
\paragraph{Term $\bm{I}$.}
Using the chain rule we get the following expression with derivatives in terms of physical coordinates instead of reference coordinates 
\begin{equation}
\hatpartial_{\hatr}^k \hatv
\sim
\sum_{j=1}^k \sum_{\{q_\ell\}\in \mcQ_j^k}
%\left(
\partial_r^j v \cdot
\prod_{\ell=1}^j 
\hatpartial_{\hatr}^{q_\ell} r
%\right)
\end{equation} 
where $\mcQ_j^k$ consists of all sets $\{q_\ell\}_{\ell=1}^j$ of positive integers such that $\sum_{\ell=1}^j q_\ell=k$.
Taking $q_\ell$ derivatives of $r$ with respect to $\hatr$ we get
\begin{equation} \label{eq:inner-der}
\hatpartial_{\hatr}^{q_\ell} r
=
\left( \prod_{i=1}^{q_\ell} (\gamma - (i-1)) \right) \hatr^{\gamma-q_\ell}
\sim
r^{(\gamma-q_\ell)/\gamma}
\end{equation}
with a constant depending only on $\gamma$.
This gives the equivalence
\begin{align}
\sum_{\{q_\ell\}\in \mcQ_j^k}
\partial_r^j v \cdot
\prod_{\ell=1}^j 
\hatpartial_{\hatr}^{q_\ell} r
&\sim
\sum_{\{q_\ell\}\in \mcQ_j^k}
\partial_r^j v \cdot
\prod_{\ell=1}^j
r^{(\gamma-q_\ell)/\gamma}
\\&
\sim
\sum_{\{q_\ell\}\in \mcQ_j^k}
\partial_r^j v \cdot r^{j - k/\gamma}
%\\&
=
\left|\mcQ_j^k \right|
\partial_r^j v \cdot r^{j - k/\gamma}
\end{align}
where we in the second equivalence used that $\sum_{\ell=1}^j q_\ell=k$.
In summary, we have the equivalence
\begin{equation}
\label{eq:leading-term-equiv-a}
\hatpartial_{\hatr}^k \hatv
\sim
\sum_{j=1}^k
\partial_r^j  v \cdot
r^{j - k/\gamma}
\end{equation}
For the measure we have
\begin{equation} \label{eq:measure-trans}
\hatr d\hatr d\hattheta
= r^{1/\gamma} \frac{d\hatr }{ dr } dr d\theta
=  r^{(1/\gamma - 1)} \left( \gamma r^{(\gamma - 1)/\gamma} \right)^{-1} r dr d\theta
= \gamma^{-1} \left( r^{(\gamma - 1)/\gamma}\right)^{-2} r dr d\theta
\end{equation}
Integrating Term~$I$ over the reference domain, using the equivalence \eqref{eq:leading-term-equiv-a}, and the Cauchy--Schwarz inequality we get
\begin{align} \label{eq:fin-leading-a}
\int_{\hatOmega}
\bigl(\hatpartial_{\hatr}^k \hatv \bigr)^2 \hatr \, d\hatr d\hattheta
& \lesssim
\sum_{j=1}^k
\int_\Omega
\underbrace{
\left(\partial_r^j  v \right)^2
}_{\leq |D^j v|^2}
\left(r^{j - k/\gamma}\right)^2
\left( r^{(\gamma - 1)/\gamma}\right)^{-2}
r \, dr d\theta
\\& \label{eq:fin-leading-b}
\lesssim
\sum_{j=1}^k
\left\| r^{j-k/\gamma-(\gamma-1)/\gamma} D^j v \right\|_\Omega^2
\end{align}
where we identify
\begin{equation}\label{eq:aj-ak}
\alpha_j = j-\frac{k}{\gamma}-\frac{\gamma-1}{\gamma}
= (j - k) + \alpha_k
\qquad\text{where}\qquad
\alpha_k = (k-1)\frac{\gamma-1}{\gamma}
\end{equation}

\paragraph{Term $\bm{II}$.}
%In the same manner, we now proceed with the remaining terms in \eqref{eq:terms-of-that-form}, i.e., we consider terms on the form $\hatr^{m-k} \hatpartial_{\hatr}^m \hatpartial_{\hattheta}^n \hatv$ where $0 \leq m \leq k-1$ and $1 \leq n \leq k-m$.
Again using the chain rule we get the following expression with derivatives in terms of physical coordinates instead of reference coordinates
\begin{equation}
\hatpartial_{\hatr}^m \hatpartial_{\hattheta}^n \hatv
\sim
\sum_{j=0}^m \sum_{\{q_\ell\}\in \mcQ_j^m}
%\left(
\partial_r^j \partial_\theta^n v \cdot
\prod_{\ell=1}^j 
\hatpartial_{\hatr}^{q_\ell} r
%\right)
\end{equation}
where we recall that by definition $\prod_{\ell=1}^0 
\hatpartial_{\hatr}^{q_\ell} r = 1$.
By \eqref{eq:inner-der} we have the equivalence
\begin{align}
\sum_{\{q_\ell\}\in \mcQ_j^m}
\partial_r^j \partial_\theta^n u \cdot
\prod_{\ell=1}^j 
\hatpartial_{\hatr}^{q_\ell} r
&\sim
\sum_{\{q_\ell\}\in \mcQ_j^m}
\partial_r^j \partial_\theta^n v \cdot
\prod_{\ell=1}^j
r^{(\gamma-q_\ell)/\gamma}
\\&
\sim
\sum_{\{q_\ell\}\in \mcQ_j^m}
\partial_r^j \partial_\theta^n v \cdot r^{j - m/\gamma}
%\\&
=
\left|\mcQ_j^m \right|
\partial_r^j \partial_\theta^n v \cdot r^{j - m/\gamma}
\end{align}
where we in the second equivalence used that $\sum_{\ell=1}^j q_\ell=m$.
In summary, we have the equivalence
\begin{align}
\label{eq:single-term-equiv-a}
\hatr^{m-k} \hatpartial_{\hatr}^m \hatpartial_{\hattheta}^n \hatv
&\sim
r^{(m-k)/\gamma}
\sum_{j=0}^m
\partial_r^j \partial_\theta^n v \cdot
r^{j - m/\gamma}
\\
\label{eq:single-term-equiv-b}
&
\sim
\sum_{j=0}^m
\partial_r^j \partial_\theta^n v \cdot r^{j - k/\gamma}
\\
\label{eq:single-term-equiv-c}
&
=
\sum_{j=0}^m
r^{j - k} \partial_r^j \partial_\theta^n v \cdot r^{k(\gamma-1)/\gamma}
\end{align}
Integrating Term~$II$ over the reference domain gives
\begin{align}
\label{eq:gen-case-a}
&
\int_{\hatOmega}
\sum_{\substack{ m\geq 0 ,\, n\geq 1 \\ k \geq m + n }}
\bigl(
\hatr^{m-k} \hatpartial_{\hatr}^m \hatpartial_{\hattheta}^n \hatv
\bigr)^2 \hatr \, d\hatr d\hattheta
\\
\label{eq:gen-case-b}
&\qquad\quad
\lesssim
\int_{\hatOmega}
\sum_{\substack{ m\geq 0 ,\, n\geq 1 \\ k \geq m + n }}
\sum_{j=0}^m
\bigl( 
r^{j - k} \partial_r^j \partial_\theta^n v \cdot r^{k(\gamma-1)/\gamma}
\bigr)^2 \hatr \, d\hatr d\hattheta
\\
\label{eq:gen-case-c}
&\qquad\quad
\sim
\sum_{\substack{ m\geq 0 ,\, n\geq 1 \\ k \geq m + n }}
\int_{\hatOmega}
\bigl( 
r^{m - k} \partial_r^m \partial_\theta^n v \cdot r^{k(\gamma-1)/\gamma}
\bigr)^2 \hatr \, d\hatr d\hattheta
\\
%\label{eq:gen-case-d}
%&\sim
%\sum_{\{m,n\}\in \mcI_k}
%\int_\Omega
%\bigl| 
%r^{m - k} \partial_r^m \partial_\theta^n v \cdot r^{k(\gamma-1)/\gamma}
%\bigr|^2
%\left( r^{(\gamma - 1)/\gamma}\right)^{-2} r \, dr d\theta
%\\
\label{eq:gen-case-e}
&\qquad\quad
\sim
\int_\Omega
\underbrace{
\sum_{\substack{ m\geq 0 ,\, n\geq 1 \\ k \geq m + n }}
\bigl( 
r^{m - k} \partial_r^m \partial_\theta^n v \bigr)^2
}_{\leq |D^k v|^2}
%\left( r^{(k-1)(\gamma-1)/\gamma} \right)^2
\left( r^{k(\gamma-1)/\gamma} \right)^2
\left( r^{(\gamma - 1)/\gamma}\right)^{-2}
r \, dr d\theta
\\
\label{eq:gen-case-f}
&\qquad\quad
\leq
\| r^{\alpha_k} D^k v \|_\Omega^2
\end{align}
with $\alpha_k$ as in \eqref{eq:aj-ak}.
Here we in \eqref{eq:gen-case-b} used the equivalence \eqref{eq:single-term-equiv-a}--\eqref{eq:single-term-equiv-c} and the Cauchy-Schwarz inequality to move the sum over $j$ outside the square; in \eqref{eq:gen-case-c} we used the fact that all index pairs $\{j,n\}$ generated by the inner sum over $j$ are also generated by the outer sum as index pairs $\{m,n\}$; and in \eqref{eq:gen-case-e} we used the transformation of the measure \eqref{eq:measure-trans}.

\paragraph{Conclusion.}
From \eqref{eq:fin-leading-a}--\eqref{eq:fin-leading-b} and \eqref{eq:gen-case-a}--\eqref{eq:gen-case-f} we have
\begin{align}
\| \hatD^k \hatv \|^2_{\hatOmega} 
&=
%\int_{\hatOmega}
%(\hatD^k \hatv )^2
%\hatr \, d\hatr d\hattheta
%\\&=
\int_{\hatOmega}
(\hatpartial_{\hatr}^k \hatv )^2
\hatr \, d\hatr d\hattheta
+
\int_{\hatOmega}
\sum_{\substack{ m\geq 0 ,\, n\geq 1 \\ k \geq m + n }}
\left(
\hatr^{m-k} \hatpartial_{\hatr}^m \hatpartial_{\hattheta}^n \hatv
\right)^2
\hatr \, d\hatr d\hattheta
%\\&
\lesssim
\sum_{j=1}^k \left\| r^{\alpha_j} D^j v \right\|^2_\Omega
\label{eq:conclusion}
\end{align}
with $\alpha_j$ given by \eqref{eq:aj-ak}. %This concludes the proof.
\end{proof}

\begin{rem}[The Scaling Parameter $\boldsymbol\gamma$]\label{remark:gamma-cond}
Combining the condition \eqref{eq:power-condition} for $\us$ to be in the weighted Sobolev space  with the expression for $\alpha_j$ \eqref{eq:alphaj-alphak} gives the inequality
\begin{equation}
j-1 - \frac{\pi}{\omegac} < \alpha_j
= j-k + \alpha_k
= j-1 - \frac{k - 1}{\gamma}
\quad\Rightarrow\quad
\gamma > \frac{(k-1) \omega_c}{\pi}
\end{equation}
Choosing the scaling parameter $\gamma$ according to this inequality means that the pullback of the solution $\hatu$, including the singular part, is in $H^k(\hatOmega)$.
\end{rem}

\section{The Parametric Cut Finite Element Method} \label{section:the-method}
We will now construct a parametric cut finite element method based on the map 
$F_\gamma$. Since the derivative of the mapping is zero in the origin for $\gamma>1$ 
we define the mesh parameters in the Nitsche forms using the well defined mesh 
parameter in the reference domain where we use a quasi-uniform mesh.

\subsection{The Mesh and Finite Element Space} \label{section:mesh-and-fem-space}

\begin{itemize}
\item Let $F: \hatOmega \rightarrow \Omega$ with $\hatOmega \subset [-1,1]^2 \subset \IR^2$ 
where $\hatOmega$ is the reference domain such that the center $(0,0) \in \partial\hatOmega$ corresponds to the singular point. For simplicity we here assume $F=F_\gamma$ but in general the mapping $F$ is a composite mapping $F=F_* \circ F_\gamma$ where the additional bijective mapping $F_*$ is readily incorporated in the method below using standard techniques.

\item Let $\hatmcK_{h,0}$ be a quasi-uniform mesh on $[-1,1]^2$ consisting of shape regular elements with mesh parameter $h$. Typically we employ uniform quadrilaterals. Let $\hatmcK_h = \{ K \in \hatmcK_{h,0} : K \cap \Omega \neq \emptyset \}$ be the active mesh in the 
reference domain.
Let $\mcK_h = F(\hatmcK_h)$ be the induced active mesh in the physical domain. 

\item Let 
$\widehat{\mcF}_h$ be the set of interior faces in $\hatmcK_h$ such that each face in $\widehat{\mcF}_h$ is adjacent to at least one element $K \in\hatmcK_h$ that is cut by the boundary $\partial\hatOmega$, i.e., $\overline{K}\cap\partial\hatOmega \neq \emptyset$.

\item Let $\widehat{V}_{h,0}$ be a finite element space consisting of continuous piecewise 
polynomials of order $p$ defined on $\hatmcK_{h,0}$ and let 
$\widehat{V}_h = \widehat{V}_{h,0} |_{\hatmcK_h}$ be the active finite element space in the 
reference domain. The active finite element space in the physical domain is given by $V_h = \hatV_h \circ F^{-1}$.

\item It is noted in the analysis below, see Remark~\ref{remark:choice-of-gamma}, that by choosing $\gamma = 2 p$ in the mapping we obtain optimal order convergence for our method regardless of the opening angle $\omega_c \in (\pi,2\pi)$.
\end{itemize}

\subsection{The Method}

\paragraph{The Method in Physical Cartesian Coordinates.}

Find $u_h \in V_h$ such that 
\begin{equation} \label{eq:method-cartesian}
A_h(u_h,v) = l_h(v) \,,\qquad \forall v \in V_h
\end{equation}
where 
\begin{align}
A_h(v,w) &= a_h(v,w) + s_h(\pi_h v,\pi_h w)
\\
a_h(v,w) &= (\nabla v,\nabla w)_{\Omega} 
-(n\cdot \nabla v,w)_{\partial \Omega}
+(v, \beta h_\Omega^{-1} w - n\cdot \nabla w)_{\partial \Omega}
\\ \label{eq:sh}
s_h(v,w) &= \hats_h(\hatv,\hatw)
=\sum_{j=1}^p \tau {h}^{2j - 1} ([\hatD^j \hatv],[\hatD^j \hatw])_{\widehat{\mcF}_h}
\\
l_h(v) &= (f,v)_{\Omega} + (g,\beta h_\Omega^{-1} v - n\cdot \nabla v )_{\partial \Omega}
\end{align}
Here $[\cdot]$ denotes the jump between neighboring elements; $\pi_h$ is a projection operator onto $V_h$, which we define in the analysis below;
and $h_\Omega \sim h \hatr^{\gamma-1}  = h r^{\frac{\gamma-1}{\gamma}}$ is a mesh function defined in such a way that 
\begin{equation}
\hata_{h} (\hatv,\hatw) = a_h(v,w), \qquad \hatl_{h}(\hatv) = l_h (v)
\end{equation}
where the reference forms $\hata_{h}$ and $\hatl_{h}$ are given below.
Note that since $\pi_h$ is a projection, i.e., $\pi_h v = v$ for $v\in V_h$, the method \eqref{eq:method-cartesian} is unaffected by its presence.
The stability form $s_h$ is naturally defined in the reference domain since we, in the analysis below, prove coercivity in the reference domain. It is furthermore easier to evaluate higher-order derivatives in the reference domain.

\paragraph{Galerkin Orthogonality.} Due to the inclusion of the projection operator $\pi_h$ the method is not consistent yielding the perturbed Galerkin orthogonality 
\begin{equation}\label{eq:gal-ort}
A_h(u-u_h,v) - s_h(\pi_h u,v) = 0 \,,\qquad \forall v \in V_h
\end{equation}
\begin{rem}[Inconsistency]
This inconsistency is artificial in the sense that for sufficiently regular $u$ the interpolation operator $\pi_h$ is not needed and $s_h(u,v)=0$. However, we include $\pi_h$ in $s_h$ so that the form $A_h$ is well defined also for exact solutions of lower regularity, which is the case of the solution to the dual problem in the proof of the $L^2$ estimate below.
\end{rem}

\paragraph{The Method in Reference Cartesian Coordinates.} 
Transforming back to Euclidian coordinates  $(\hatx,\haty) = (\hatr \cos \hattheta,\hatr \sin \hattheta )$ 
in the reference domain we obtain the method
\begin{equation} \label{eq:method-ref}
\hatA_{h}(\hatu_{h},\hatv) = \hatl_{h}(\hatv) \,,\qquad \forall \hatv\in\hatV_h
\end{equation}
where $\hatA_{h}(\hatv,\hatw) = \hata_{h}(\hatv,\hatw) + \hats_{h}(\hatv,\hatw)$. The reference forms are defined by \eqref{eq:sh} and
\begin{align} \label{eq:ref-forms-a}
\hata_h(\hatv,\hatw) 
&= \left( \hatnabla \hatv, B \hatnabla \hatw \right)_{\hatOmega} 
- \left( \hatn \cdot B \hatnabla \hatv, \hatw \right)_{\partial \hatOmega}
+ \left( \hatv, \beta h^{-1} - \hatn \cdot B \hatnabla \hatw \right)_{\partial \hatOmega}
\\ \label{eq:ref-forms-b}
\hatl_{h}(\hatv) &= \left(\gamma\left(\hatx^2 + \haty^2 \right)^{\gamma-1}\hatf,\hatv \right)_{\hatOmega} 
 + \left( \hatg,\beta {h}^{-1} \hatv - \hatn \cdot B\hatnabla 
\hatv \right)_{\partial \hatOmega}
\end{align}
where $\hatn$ is the outward pointing unit normal to $\partial\hatOmega$ in Cartesian reference coordinates and
we employed the notation
\begin{equation}
B = S_{\hattheta}^T D_\gamma S_{\hattheta}
\qquad\text{with}\qquad
D_\gamma 
= \left(
\begin{matrix}
\gamma^{-1} &  0
\\
0 & \gamma 
\end{matrix}
\right)
,\
S_{\hattheta}
=
\begin{pmatrix}
 \cos \hattheta &  \sin \hattheta
\\
-\sin \hattheta & \cos \hattheta
\end{pmatrix}
\end{equation}
See Appendix~\ref{appendix:riemann} for details of the derivation. We note that $B$ is symmetric and positive definite
\begin{equation}\label{eq:B-posdef}
\| \xi \|_{\IR^2}^2 \leq c_\gamma (B\xi,\xi)_{\IR^2}
\,,\qquad \forall \xi \in \IR^2
\end{equation} 
with $c_\gamma = \min(\gamma,\gamma^{-1})$,
and uniformly bounded
\begin{equation}\label{eq:B-cont}
(B\xi,\eta)_{\IR^2} \leq C_\gamma \| \xi \|_{\IR^2} \| \eta \|_{\IR^2} \,,\qquad \forall \xi,\eta \in \IR^2
\end{equation}
with $C_\gamma = \max(\gamma,\gamma^{-1})$.
Furthermore, since $S_{\hattheta}$ is a rotation matrix the identity $S_{\hattheta} S_{\hattheta}^T = I$ holds and also $B^m$ is uniformly bounded
\begin{equation}\label{eq:B-cont-n}
(B^m\xi,\eta)_{\IR^2} \leq C_\gamma^m \| \xi \|_{\IR^2} \| \eta \|_{\IR^2} \,,\qquad \forall \xi,\eta \in \IR^2
\,,\
m=1,2,\dots 
\end{equation}

\begin{rem}[Numerical Stability]
By implementing the method in reference coordinates according to \eqref{eq:method-ref} we conveniently avoid numerical issues induced by the mapping $F_\gamma$ having unbounded derivatives when approaching the singular point.
Even though the forms in the physical respectively reference domains are mathematically identical, evaluation of the forms in the physical domain involves products between terms tending to infinity and to zero when approaching the singular point, causing numerical accuracy and stability issues. Thanks to simplification, this issue does not exist in the reference domain forms \eqref{eq:ref-forms-a} and \eqref{eq:ref-forms-b}.
\end{rem}

\subsection{The Method with Multiple Nonconvex Corners} \label{section:multipatch}
We consider situations where the domain $\Omega$ is partitioned into 
a finite set of patches $\{\Omega_i\}_{i=0}^n$ in such a way that each singular corner belongs to one patch as illustrated in Figure~\ref{fig:multipatch}. The scaling 
can then be applied locally to each patch $\Omega_i$ containing a singular corner. The coupling between the patches 
is done weakly using Nitsche's method which does not require the meshes to match across the 
interface, see the parametric multipatch method in \cite{JoLaLa2017}. This approach enables us to use the grading locally in the vicinity of each corner without difficulty.

\paragraph{Multipatch Meshes and Finite Element Spaces.} 
For each patch $\Omega_i$ we construct a mapping $F^{(i)}$, a mesh $\mcK_h^{(i)}$ and a finite element space $V_h^{(i)}$ as outlined in Section~\ref{section:mesh-and-fem-space}. On the complete multipatch domain $\Omega$ we define the finite element space
\begin{equation}
V_h = \bigoplus_{i=0}^n V_h^{(i)}
\end{equation}
where we note that $V_h$ is discontinuous across patch interfaces. This will instead be enforced weakly in the method.

\paragraph{The Multipatch Method in Physical Cartesian Coordinates.} Defining the domain $\Omega = \bigcup_{i=0}^n \Omega_i$
and the patch interface $\Gamma = \bigcup_{i<j} \overline{\Omega}_i \cap \overline{\Omega}_j$ we adapt formulation \eqref{eq:method-cartesian} to the multi-patch setting by redefining $a_h$ and $s_h$ as
\begin{align}
\label{eq:interface-ah}
a_{h}(v,w) &= (\nabla v,\nabla w)_{\Omega \setminus \Gamma} 
-(n\cdot \nabla v,w)_{\partial\Omega}
+(v, \beta h_\Omega^{-1}w  - n\cdot \nabla w )_{\partial\Omega}
\\ \nonumber
&\qquad 
-(\langle n\cdot \nabla v \rangle , [w])_{\Gamma}
+([v], \beta \langle h_\Omega^{-1} \rangle [w] - \langle n\cdot \nabla w \rangle )_{\Gamma}
\\
\label{eq:interface-sh}
s_{h}(v,w) & = \sum_{i=0}^n \sum_{j=1}^k \tau {h}^{2j - 1} ([\hatD^j \hatv],[\hatD^j \hatw])_{\widehat{\mcF}_{i,h}}
\end{align}
In \eqref{eq:interface-ah} we have added Nitsche interface terms that couple the patches and \eqref{eq:interface-sh} now includes Ghost penalty stabilization on the computational mesh of each patch. In the interface terms $\langle\cdot\rangle$ denotes the average between adjoining patches while $[\cdot]$ denotes the jump.
Expressing the interface terms in reference Cartesian coordinates is straightforward and we refer to \cite{JoLaLa2017} for further details.
  
\begin{figure}
\centering
\begin{subfigure}[t]{0.45\linewidth}\centering
\includegraphics[height=0.42\linewidth]{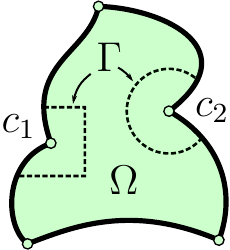}
\subcaption{Domain with two nonconvex corners}
\label{fig:multipatch-a}
\end{subfigure}
\begin{subfigure}[t]{0.45\linewidth}\centering
\includegraphics[height=0.42\linewidth]{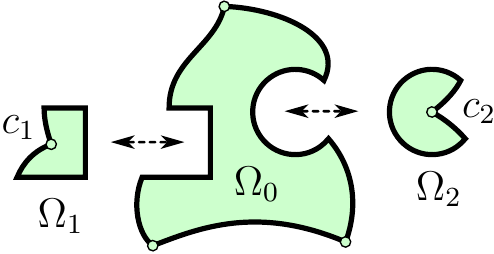}
\subcaption{Partition of domain}
\label{fig:multipatch-b}
\end{subfigure}
\centering
\caption{Partition of a domain with two nonconvex corners into patches, $\Omega = \cup_{i=0}^2 \Omega_i$, such that no patch contains more than one of the original nonconvex corners. Note that nonconvex corners on the patch interface induced by the partition do not give rise to new singularities.
}
\label{fig:multipatch}
\end{figure}

\section{A Priori Error Estimates} \label{section:error-est}
\paragraph{Main Approach.}
\begin{itemize}
\item We start by defining an energy norm associated with the form $\hatA_{h}$ in the 
reference domain and we establish coercivity and continuity using standard arguments 
since the mesh is quasi-uniform and $B$ is positive definite and bounded, see (\ref{eq:B-posdef}) 
and (\ref{eq:B-cont}).
\item Next we define a corresponding energy norm in the physical domain using the mapping which 
leads to the corresponding coercivity and continuity results for the physical problem.
\item We define an interpolation operator in the reference domain and recall a standard interpolation error 
estimate. This then leads to an interpolation error estimate in the physical domain using 
(\ref{eq:norm-eq-Hkalpha}). 
\item Using the above results optimal order a priori results follows using standard techniques.
\end{itemize}
\subsection{Basic Results}

\paragraph{Extension Operator.}
Let $U(\Omega)\subset \IR^2$ be a neighborhood of $\Omega$ such that 
$\mcK_h \subset U(\Omega)$ for all $h\in(0,h_0]$. Due to the fact that $\omegac 
< 2\pi$ there is an extension operator $\widehat{E}: H^s (\hatOmega) 
\rightarrow  H^s(\widehat {U(\Omega)} )$ that satisfies the stability estimate 
\begin{equation}
\| \widehat{E}(\hatv) \|^2_{H^s(\widehat{U(\Omega)})} 
\lesssim
\| \hatv \|_{H^s(\hatOmega)}
\end{equation}

\paragraph{Interpolant.}
Let 
$\hatpi_{h} : H^1(\widehat{U(\Omega)}) \rightarrow \widehat{V}_h$ be a Scott--Zhang 
type interpolation operator and recall that we have the error estimate
\begin{equation} \label{eq:Hm-interpolation}
\| \hatv - \hatpi_{h} \hatv \|_{H^m(\hatmcK_h)} 
\lesssim h^{s-m} | \hatv |_{H^{s}(\widehat{U(\Omega)})}\,, 
\qquad 0\leq m \leq s \leq p+1
\end{equation}
In the physical domain we define $\pi_h : H^1(U(\Omega)) \rightarrow V_h$ by
\begin{equation}
\widehat{\pi_h v} = \widehat{\pi}_{h} \hatv 
\end{equation}
Since the Scott--Zhang interpolation operator is a projection, this is also the case for $\pi_h$.

\begin{rem}[Choice of Scaling Parameter $\boldsymbol\gamma$]\label{remark:choice-of-gamma}
To obtain optimal order interpolation estimates we need control over $s=p+1$ derivatives in the reference domain.
A condition on $\gamma$ for obtaining full regularity of the pullback of the solution in the reference domain was given in Remark~\ref{remark:gamma-cond} and choosing $k=p+1$ therein we arrive at the condition
\begin{equation}\label{eq:gamma-condition-p}
\gamma > \frac{p\omegac}{\pi}
\end{equation}
As $\omegac \in (\pi,2\pi)$ a choice that holds for any opening angle is $\gamma = 2p$.
\end{rem}

\paragraph{Stabilization Form.}
By construction the stabilization form \eqref{eq:sh} satisfies the following abstract properties:
\begin{itemize}
\item The form is consistent, i.e.,
\begin{equation} \label{eq:sh-consistent}
\hats_h(\hatv,\hatv)=0
\,,\qquad\forall \hatv\in H^{p+1}(\hatOmega)
\end{equation}
\item The form satisfies the interpolation estimate
\begin{equation} \label{eq:sh-interpolation}
\hats_h(\hatpi_h\hatv,\hatpi_h\hatv) \lesssim h^{2(s-1)} |\hatv |_{H^{s}(\widehat{U(\Omega)})}^2
\,,\qquad\forall \hatv\in H^{s}(\hatOmega)
\,,\ 1 \leq s \leq p+1
\end{equation}

\item The following inverse inequality holds
\begin{equation} \label{eq:sh-inverse-ineq}
\| h^{1/2} \hatnabla\hatv\|_{\partial\hatOmega}^2
\lesssim
\| \hatnabla\hatv \|_{\hatOmega}^2 + \hats_h(\hatv,\hatv)
\,,\qquad\forall \hatv \in \hatV_h
\end{equation}
\end{itemize}

\paragraph{Energy Norm.}
Define the following energy norms associated with the Nitsche method
\begin{equation}\label{eq:energynorm-ref}
\tn \hatv \tn_{h}^2 
= \| \hatnabla \hatv \|^2_{\hatOmega} 
+  \| h^{1/2} \hatnabla \hatv \|^2_{\partial\hatOmega} 
+ \| h^{-1/2} \hatv \|^2_{\partial\hatOmega} + \| \hatpi_h \hatv \|^2_{\hats_h}
\end{equation}
and the push forward energy norm
\begin{equation}\label{eq:energynorm-phys}
\tn v \tn_{h,\Omega} = \tn \hatv \tn_{h}
\end{equation}
%Note that the notation for both norms is identical and distinction between the norms is implied by the argument being a function of physical respectively reference coordinates.

\paragraph{Interpolation in Energy Norm.}
For $\hatv\in H^{s}(\hatOmega)$ we on every $K \in \hatmcK_h$ have $(\hatv - \hatpi_h\hatv)|_K \in H^{s}(K)$.
Assuming $s\geq 2$ and using the element-wise trace inequality
\begin{equation}
\| \hatw \|^2_{\partial\hatOmega \cap K} \lesssim h^{-1} \| \hatw \|^2_{K} + h \| \hatnabla \hatw \|^2_{K} \,,\qquad\forall \hatw|_{K} \in H^1(K)\,,\ \forall K\in\hatmcK_h
\end{equation}
which holds independent of the position of $\partial\hatOmega$ in $K$, see \cite{HanHanLar2003}, and the interpolation estimate \eqref{eq:Hm-interpolation}
%, and the properties \eqref{eq:sh-consistent} and \eqref{eq:sh-interpolation} of the stabilization form $\hats_h$,
we obtain
\begin{equation}\label{eq:interpol-energy-norm-ref}
 \tn \hatv - \hatpi_{h} \hatv \tn_{h} 
\lesssim h^{s-1} \| \hatD^{s} \hatv \|_{\widehat{U(\Omega)}} \,,\qquad 1 \leq s \leq p+1
\end{equation}
Using the definition of the push forward energy norm (\ref{eq:energynorm-phys}), the interpolation estimate 
in the reference domain (\ref{eq:interpol-energy-norm-ref}), and the estimate (\ref{eq:norm-eq-Hkalpha}), we obtain the following interpolation estimate in the physical domain
\begin{equation}\label{eq:interpol-energy-norm-phys}
 \tn v - \pi_h v \tn_{h,\Omega} 
\lesssim 
h^{s-1} 
\| v \|_{H^{s,\alpha}(\Omega)} \,,\qquad 1 \leq s \leq p+1
\end{equation}
More precisely we proceeded as follows
\begin{equation}
\tn v -\pi_h v \tn_{h,\Omega} 
= 
\tn \hatv - \widehat{\pi_h v} \tn_{h}
=
\tn \hatv - \hatpi_{h} {\hatv} \tn_{h}
\lesssim 
h^{s-1} \| \hatD^{s}\hatv \|_{\widehat{U(\Omega)}}
\lesssim 
h^{s-1} \| v \|_{H^{s,\alpha}(\Omega)}
\end{equation}

\paragraph{Continuity and Coercivity.}
It holds
\begin{equation}\label{eq:continuity-ref}
\hatA_h(\hatv,\hatw) \lesssim \tn \hatv \tn_h \tn \hatw \tn_h \,, \qquad \hatv,\hatw \in H^{3/2+\epsilon}(\hatOmega) + \hatV_h
\end{equation}
and for $\beta>0$ large enough
\begin{equation}\label{eq:coercivity-ref}
\tn \hatv \tn_h^2 \lesssim \hatA_h(\hatv,\hatv) \,, \qquad \hatv \in \hatV_h
\end{equation}
These results follows directly using standard arguments for cut finite elements 
since the mesh is quasi-uniform in the reference domain and $B$ is positive definite 
and bounded, see (\ref{eq:B-posdef}) and (\ref{eq:B-cont}).
\begin{proof}
\textbf{(\ref{eq:continuity-ref}).}
Using the triangle and the Cauchy--Schwarz inequalities, and the bound \eqref{eq:B-cont-n} on $B^2$ we readily get
\begin{align}
| \hatA_h(\hatv,\hatw) | &\leq
\| B \hatnabla \hatv \|_{\hatOmega} \| \hatnabla \hatw \|_{\hatOmega}
\\&\quad\nonumber
+
\| h^{1/2} B \hatnabla \hatv \|_{\partial\hatOmega}
\| h^{-1/2} \hatw \|_{\partial\hatOmega}
+
\| h^{-1/2} \hatv \|_{\partial\hatOmega}
\| h^{1/2} B \hatnabla \hatw \|_{\partial\hatOmega}
\\&\quad\nonumber
+
\beta \| h^{-1/2} \hatv \|_{\partial\hatOmega}
\| h^{-1/2} \hatw \|_{\partial\hatOmega}
+
\| \hatpi_h \hatv \|_{\hats_h}
\| \hatpi_h \hatw \|_{\hats_h}
\\
&\leq
C_\gamma^2
\Bigl(
\| \hatnabla \hatv \|_{\hatOmega} \| \hatnabla \hatw \|_{\hatOmega}
\\&\qquad\quad\nonumber
+
\| h^{1/2} \hatnabla \hatv \|_{\partial\hatOmega}
\| h^{-1/2} \hatw \|_{\partial\hatOmega}
+
\| h^{-1/2} \hatv \|_{\partial\hatOmega}
\| h^{1/2} \hatnabla \hatw \|_{\partial\hatOmega} \Bigr)
\\&\quad\nonumber
+
\beta \| h^{-1/2} \hatv \|_{\partial\hatOmega}
\| h^{-1/2} \hatw \|_{\partial\hatOmega}
+
\| \hatpi_h \hatv \|_{\hats_h}
\| \hatpi_h \hatw \|_{\hats_h}
\\&
\lesssim
\tn \hatv \tn_h \tn \hatw \tn_h
\end{align}
with a hidden constant $c = C_\gamma^{2} + \beta$.

\paragraph{(\ref{eq:coercivity-ref}).}
Using the Cauchy--Schwarz inequality on the consistency term, and applying a Young's inequality $2ab \leq \delta a^2 + \delta^{-1} b^2$ with $\delta>0$, we get the bound
\begin{align}  \label{eq:const-bound-a}
2\bigl| (\hatn \cdot B \hatnabla\hatv, \hatv)_{\partial\hatOmega} \bigr|
&\leq
2\| h^{1/2} B \hatnabla\hatv \|_{\partial\hatOmega} \| h^{-1/2} \hatv \|_{\partial\hatOmega}
\\
&\leq
\delta \| h^{1/2} B \hatnabla\hatv \|_{\partial\hatOmega}^2
+
\delta^{-1} \| h^{-1/2} \hatv \|_{\partial\hatOmega}^2
\\ \label{eq:const-bound-c}
&\leq
C_\gamma^2 \delta \| h^{1/2} \hatnabla\hatv \|_{\partial\hatOmega}^2
+
\delta^{-1} \| h^{-1/2} \hatv \|_{\partial\hatOmega}^2
\\&\leq  \label{eq:const-bound-d}
c_{I} C_\gamma^2 \delta \left( \| \hatnabla\hatv \|_{\hatOmega}^2 + \hats_h(\hatv,\hatv) \right)
+
c_{I} \delta^{-1} \| h^{-1/2} \hatv \|_{\partial\hatOmega}^2
\end{align}
where we in \eqref{eq:const-bound-c} use that $B^2$ is bounded via \eqref{eq:B-cont-n} and in the last step utilize the inverse inequality \eqref{eq:sh-inverse-ineq} with hidden constant $c_I$, which is possible since $\hatv\in\hatV_h$.
By the positive definiteness of $B$ \eqref{eq:B-posdef} and the bound \eqref{eq:const-bound-a}--\eqref{eq:const-bound-d} we now obtain
\begin{align}
\hatA_h(\hatv,\hatv)
&=
(\hatnabla\hatv,B \hatnabla\hatv)_{\hatOmega}
-2 (\hatn \cdot B \hatnabla\hatv, \hatv)_{\partial\hatOmega}
+ (\beta h^{-1} \hatv, \hatv)_{\partial\hatOmega}
+ \hats_h(\hatv,\hatv)
\\&\geq c_\gamma\| \hatnabla\hatv \|_{\hatOmega}^2
-2 \bigl| (\hatn \cdot B \hatnabla\hatv, \hatv)_{\partial\hatOmega} \bigr|
+ \beta \| h^{-1/2}\hatv \|_{\partial\hatOmega}^2
+ \hats_h(\hatv,\hatv)
\\&\geq \label{eq:Ah-bound-c}
(c_\gamma - c_I C_\gamma^2 \delta) \left( \| \hatnabla\hatv \|_{\hatOmega}^2 + \hats_h(\hatv,\hatv) \right)
+
(\beta - c_I \delta^{-1}) \| h^{-1/2}\hatv \|_{\partial\hatOmega}^2
\end{align}
Choosing $\delta < c_\gamma c_I^{-1} C_\gamma^{-2}$ and $\beta > c_I \delta^{-1}$ yield positive constants in \eqref{eq:Ah-bound-c}. The proof is completed by transforming part of the $\left( \| \hatnabla\hatv \|_{\hatOmega}^2 + \hats_h(\hatv,\hatv) \right)$ term into a $\| h^{1/2} \hatnabla\hatv\|_{\partial\hatOmega}^2$ term via the inverse inequality \eqref{eq:sh-inverse-ineq}.
\end{proof}

Analogously, in physical coordinates it holds
\begin{equation}\label{eq:continuity-phys}
A_h(v,w) \lesssim \tn v \tn_{h,\Omega} \tn w \tn_{h,\Omega} \,, \qquad v,w \in H^{2,\alpha}(\Omega) + V_h
\end{equation}
and for $\beta$ large enough
\begin{equation}\label{eq:coercivity-phys}
\tn v \tn_{h,\Omega}^2 \lesssim A_h(v,v)\,, \qquad \forall v \in V_h
\end{equation}
These results follows using the definition of the energy norms (\ref{eq:energynorm-phys}) 
and the corresponding results (\ref{eq:continuity-ref}) and (\ref{eq:coercivity-ref}) in 
the reference domain. More precisely 
\begin{equation}
A_h(v,w) = \hatA_{h} (\hatv,\hatw) \lesssim \tn \hatv \tn_{h} \tn \hatw \tn_{h} 
=  \tn v \tn_{h,\Omega} \tn w \tn_{h,\Omega} 
\end{equation}
and 
\begin{equation}
\tn v \tn_{h,\Omega}^2 = \tn \hatv \tn_{h}^2 \lesssim \hatA_{h} (\hatv,\hatv) 
= A_h(v,v)
\end{equation}

\subsection{Error Estimates}
\begin{thm} \label{thm:error}
For $\gamma > p\omegac/\pi$ it holds 
\begin{align}
\label{eq:error-est-energy}
\tn u - u_h \tn_{h,\Omega} 
&\lesssim  {h}^p \| u \|_{H^{p+1,\alpha}(\Omega)}
\\
\label{eq:error-est-L2}
\| u - u_h \|_{\Omega} 
&\lesssim  
{h}^{p+1} \| u \|_{H^{p+1,\alpha}(\Omega)}
\end{align}
with powers $\{\alpha_j\}_{j=1}^{p+1}$ in the weighted Sobolev norm given by \eqref{eq:aj-ak} using $k=p+1$.
\end{thm}

\begin{proof} {\bf (\ref{eq:error-est-energy}).}  Using the coercivity (\ref{eq:coercivity-phys}), 
Galerkin orthogonality (\ref{eq:gal-ort}), continuity (\ref{eq:continuity-phys}), and the interpolation 
estimate (\ref{eq:interpol-energy-norm-phys}),  the estimate follows directly 
\begin{align}
\tn u - u_h \tn^2_{h,\Omega}
&
\lesssim A_h(u-u_h,u-u_h) 
\\
&= A_h(u-u_h, u- \pi_h u ) + s_h(\pi_h u, \pi_h u - u_h)
\\
&\lesssim 
\tn u-u_h \tn_{h,\Omega} \tn u- \pi_h u \tn_{h,\Omega}
+
\| \pi_h u \|_{s_h} \| \pi_h u - u_h \|_{s_h}
\\
&\lesssim
{h}^p 
\tn u-u_h \tn_{h,\Omega}  
\| u \|_{H^{p+1,\alpha}(\Omega)}
\end{align}
\paragraph{(\ref{eq:error-est-L2}).}
For the $L^2$ estimate we let $\phi\in H^{1}(\Omega)$ be the solution to the dual problem
\begin{equation}\label{eq:dual-problem}
a(v,\phi) = (\psi,v)_\Omega \,,\qquad \forall v \in H^{1}(\Omega)
\end{equation}
for which we have the weighted regularity estimate 
\begin{equation} \label{eq:dual-regularity}
\| \phi \|_{H^{2,\alpha}(U(\Omega))} 
\lesssim 
\|\psi \|_{\Omega}
\end{equation}
Setting $\psi = u-u_h$ and using the dual problem \eqref{eq:dual-problem}, Galerkin orthogonality (\ref{eq:gal-ort}), continuity (\ref{eq:continuity-phys}) and the Cauchy--Schwarz inequality, the interpolation estimates  (\ref{eq:interpol-energy-norm-phys}) and \eqref{eq:sh-interpolation}, and the regularity estimate \eqref{eq:dual-regularity},
we obtain 
\begin{align}
\| u-u_h \|^2_\Omega &= A_h(e,\phi) 
\\
&= A_h( u - u_h ,\phi - \pi_h \phi) - s_h(\pi_h u, \pi_h \phi)
 \\
&\lesssim \tn u - u_h \tn_{h,\Omega} \tn \phi - \pi_h \phi \tn_{h,\Omega}
+
\| \pi_h u \|_{s_h} \| \pi_h \phi \|_{s_h}
 \\
&\lesssim h^{p+1} \| u \|_{H^{p+1,\alpha}(U(\Omega))}  \| \phi \|_{H^{2,\alpha}(U(\Omega))} 
\\
&\lesssim 
 h^{p+1} \| u \|_{H^{p+1,\alpha}(U(\Omega))} \| \psi \|_\Omega
\\
&=
 h^{p+1} \| u \|_{H^{p+1,\alpha}(U(\Omega))} \|  u-u_h  \|_\Omega
\end{align}
which completes the proof.
\end{proof}

\section{Numerical Experiments} \label{section:numerics}

\paragraph{Implementation.}
The method was implemented using the formulation in reference Cartesian coordinates \eqref{eq:method-ref}
and other than that we follow the implementation outlined in \cite{JoLaLa2017} including the quadrature. A limitation due to our current implementation of certain geometrical operations, for example computing the intersection between an element and the domain, is that the geometry must be represented as a polygon. However, in our examples this polygon approximation is of very high resolution compared to the mesh size and should have a negligible effect on the results.

\paragraph{Approximation Space.}
As approximation space $V_h$ we use the space spanned by piecewise quadratic B-splines of maximum regularity, i.e., $V_h \subset C^1(\Omega) \subset H^2(\Omega)$. These are constructed as tensor products of 1D quadratic B-spines on a structured grid, which in general is cut by the domain. For reference purposes we also briefly consider conforming approximation spaces in the form of standard Lagrange triangular elements of various orders.

\paragraph{Parameter Values.}
For the mapping $F_\gamma$
we choose the radial scaling parameter $\gamma=2p$, as this choice holds for all opening angles $\omegac\in(\pi,2\pi)$, see Remark~\ref{remark:choice-of-gamma}.
In the method we use
the Nitsche penalty parameter $\beta=100$ and the stabilization form parameter $\tau=0.1$.

\subsection{Convergence and Stability}

\paragraph{Model Problem.}
To investigate the convergence and stability of the method we construct a model problem on a domain in the shape of a unit circle sector with various opening angles $\omegac \in (\pi,2\pi)$, see Figure~\ref{fig:unit-circle-sector}. On this domain we manufacture a problem with known singular solution by taking \eqref{eq:singular-part} as an ansatz for $u$ and deriving the corresponding right hand side $f$ and Dirichlet boundary data $g$.

\begin{figure}
\centering
\begin{subfigure}[t]{0.45\linewidth}\centering
\includegraphics[width=0.45\linewidth]{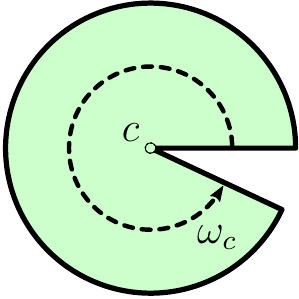}
\end{subfigure}
\centering
\caption{
Our model problem has a domain in the form of a unit circle sector with opening angle $\omegac$.
}
\label{fig:unit-circle-sector}
\end{figure}

\paragraph{Convergence.}
We here provide numerical studies of convergence using the model problem as extra validation of our theoretical a priori error estimates in Theorem~\ref{thm:error}.
First we, as a reference, apply the method to conforming finite element spaces in the form of standard Lagrange triangles of order $p=1,2,3$. These results are presented in Figure~\ref{fig:convergence-ref} where we note that the mapping indeed has the desired effect of recovering the optimal order convergence as implied by our estimates.

\begin{figure}
\centering	
	\begin{subfigure}[t]{0.32\linewidth}\centering
		\includegraphics[width=0.7\linewidth,trim=0 -65 0 0, clip]{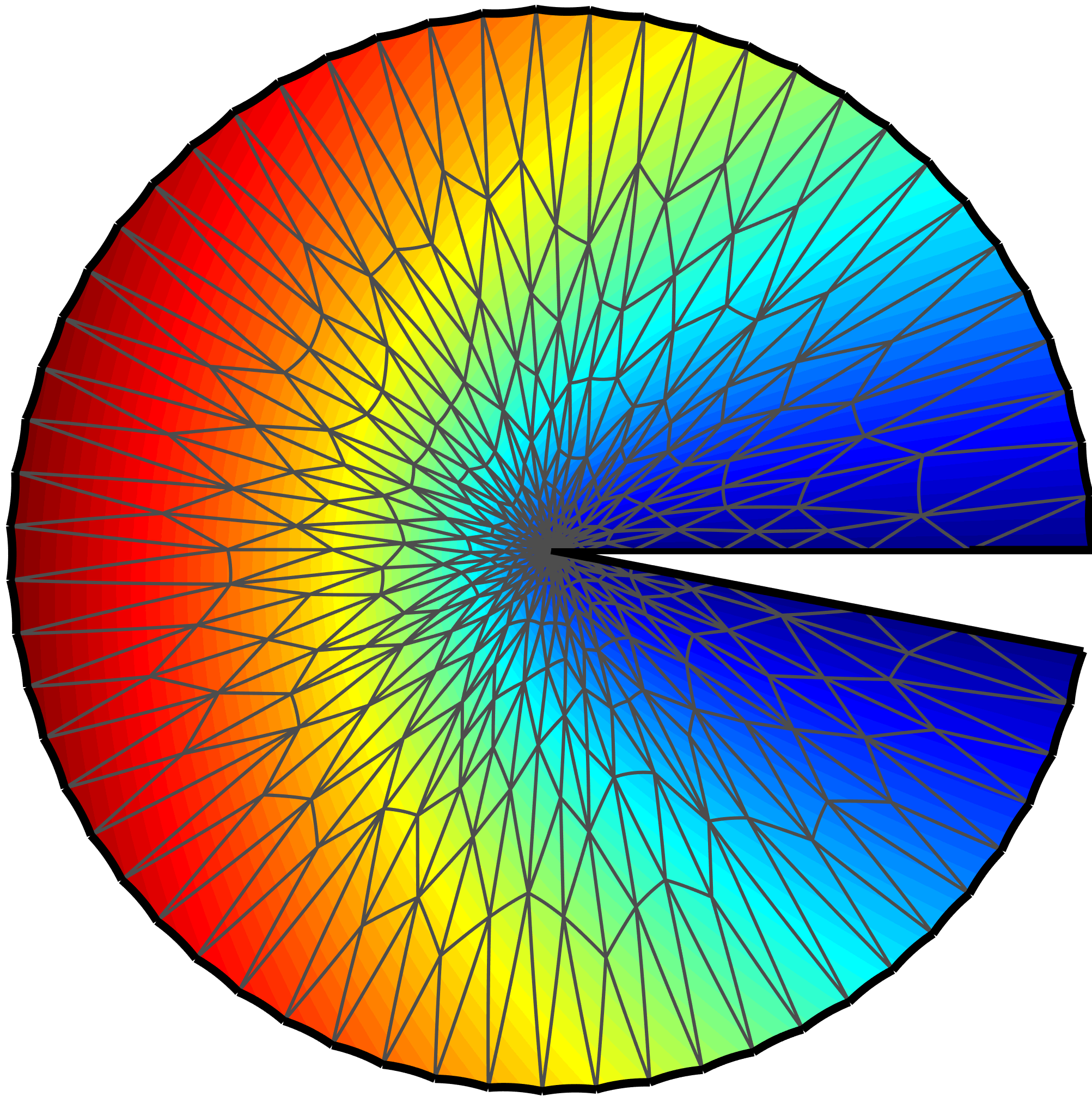}
		\subcaption{Numerical solution ($p=2$)}
		\label{fig:convergence-ref-a}
	\end{subfigure}
	\begin{subfigure}[t]{0.32\linewidth}\centering
		\includegraphics[width=0.985\linewidth]{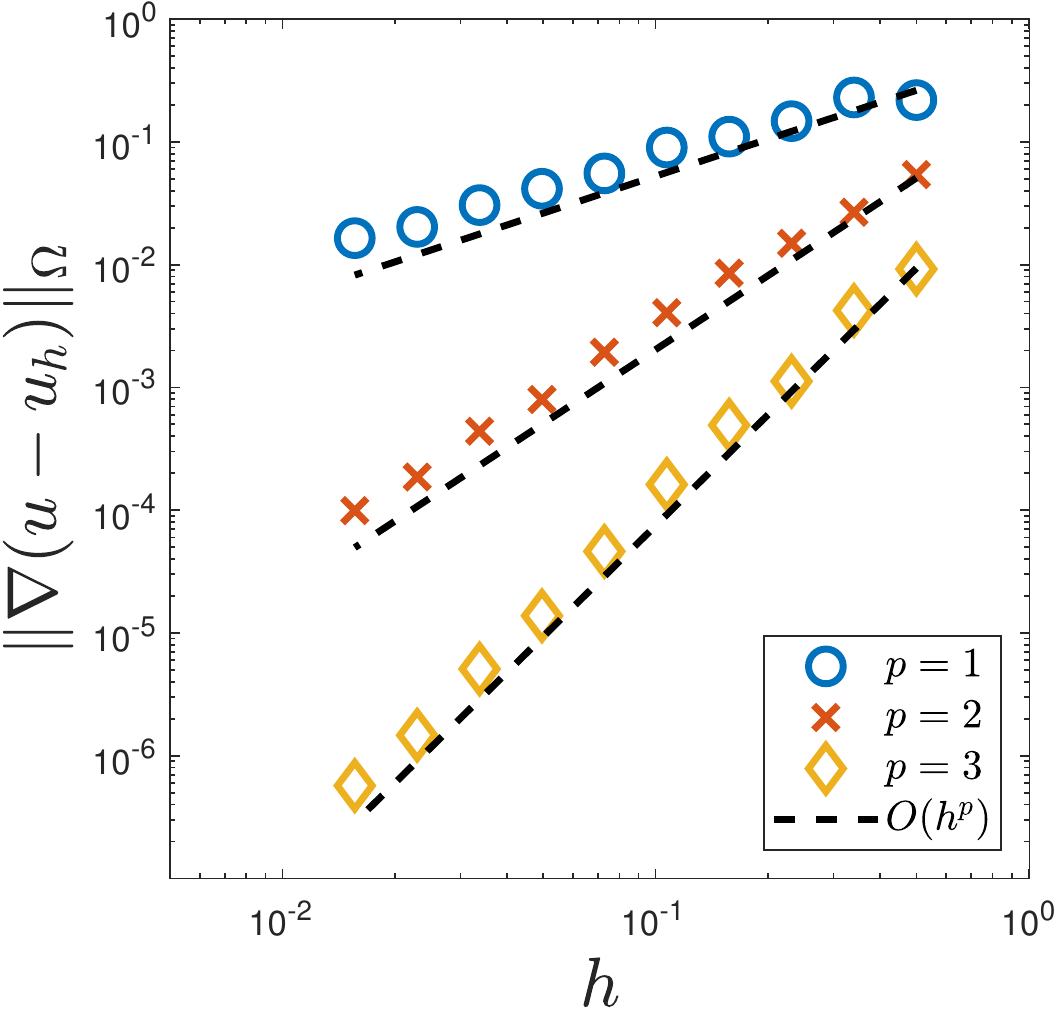}	
		\caption{Error in $H^1$ semi-norm}
		\label{fig:convergence-ref-b}
	\end{subfigure}
	\begin{subfigure}[t]{0.32\linewidth}\centering
		\includegraphics[width=1\linewidth]{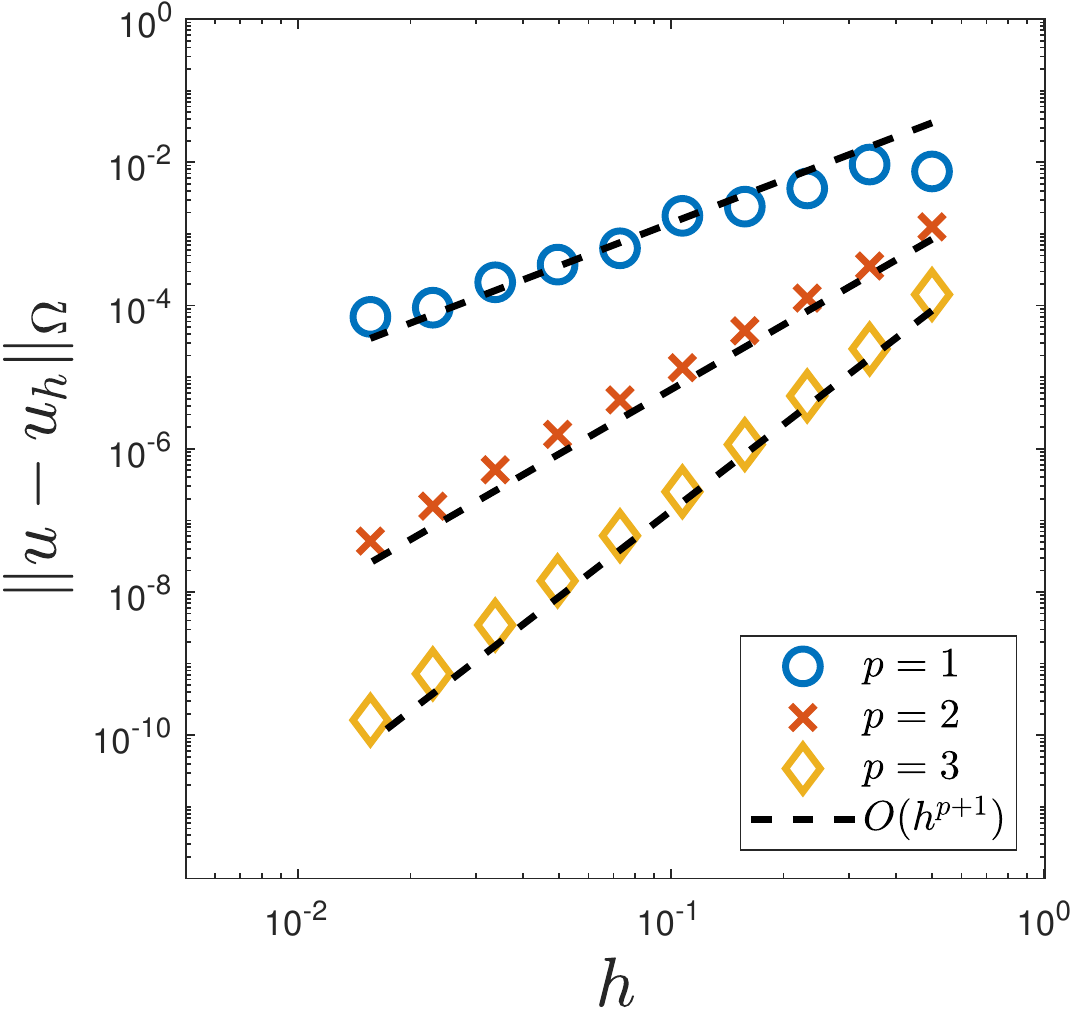}	
		\caption{Error in $L^2$ norm}
		\label{fig:convergence-ref-c}
	\end{subfigure}
\caption{
Convergence when applying the method to standard Lagrange elements on a fitted triangular mesh for the model problem with $\omegac = 0.97\times 2\pi$.
(a) Numerical solution using quadratic Lagrange elements. Note that the effect of the parametric mapping $F_\gamma$ is that the mesh, viewed in physcial coordinates, is graded towards the singularity.
(b)--(c) Convergence in $H^1(\Omega)$ semi-norm and in $L^2(\Omega)$ norm using Lagrange elements of order $p=1,2,3$. Note that optimal order convergence is recovered in all cases.
}
\label{fig:convergence-ref}
\end{figure}

Next we apply the method to cut finite element spaces in the form of $C^1$ splines, i.e., B-splines of order $p=2$, and the results are presented in Figure~\ref{fig:convergence-nosplit}. In this experiment we note an at least initial loss of optimal convergence rate when the opening angle $\omegac$ approaches $2\pi$.
Our explanation for this is that since the finite element space is constructed independently from the geometry the finite element space does not take the opening slit into account. Thus, every basis function that have  support on both sides of the slit actually provides an undesirable coupling over the slit. To remedy this we in the next paragraph outline and numerically assess a fix.
 
\begin{figure}
\centering	
	\begin{subfigure}[t]{0.32\linewidth}\centering
		\includegraphics[width=0.86\linewidth,trim=0 -25 0 0, clip]{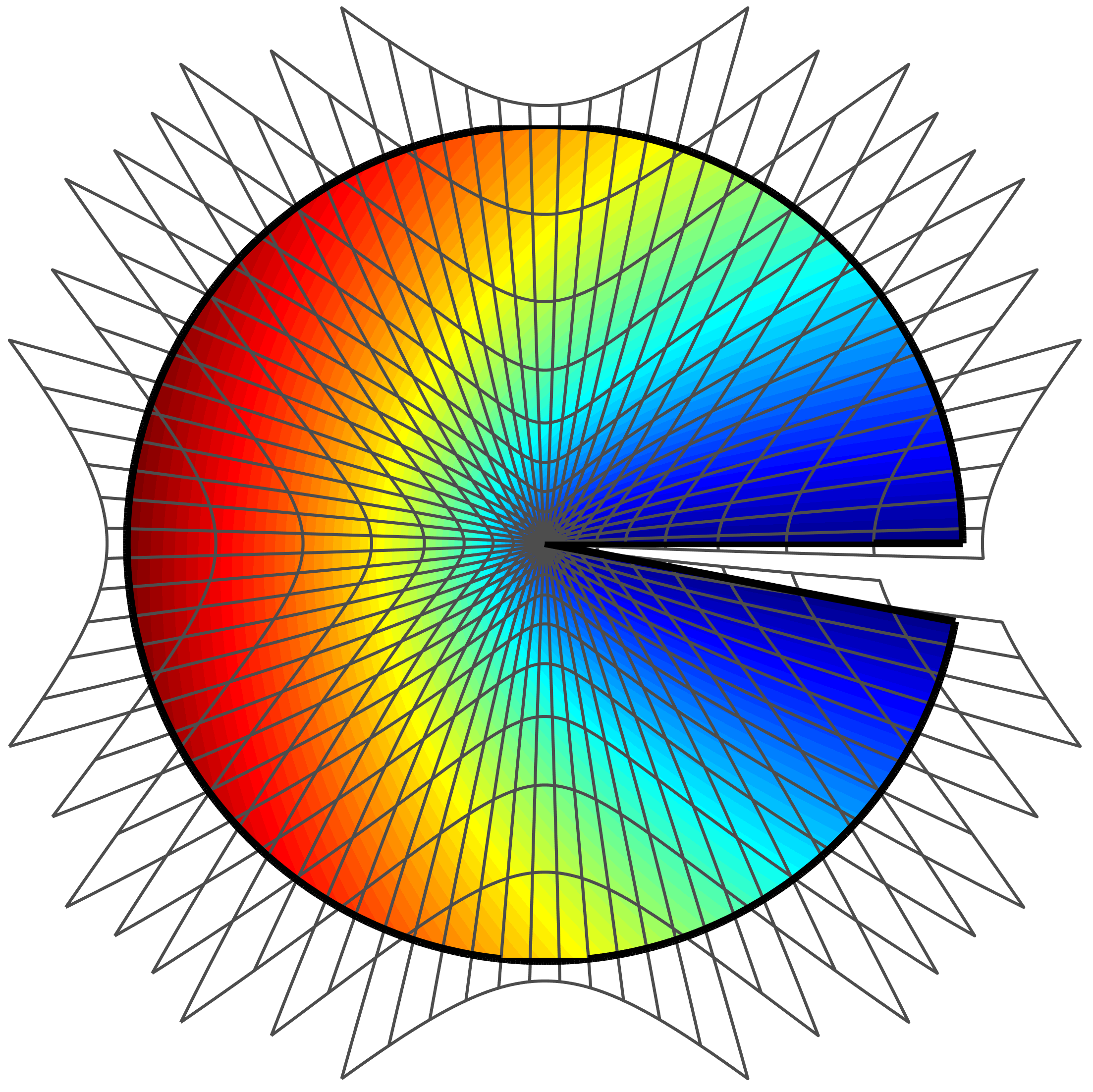}
		\subcaption{Numerical solution $C^1$ splines}
		\label{fig:convergence-nosplit-a}
	\end{subfigure}
	\begin{subfigure}[t]{0.32\linewidth}\centering
		\includegraphics[width=1\linewidth]{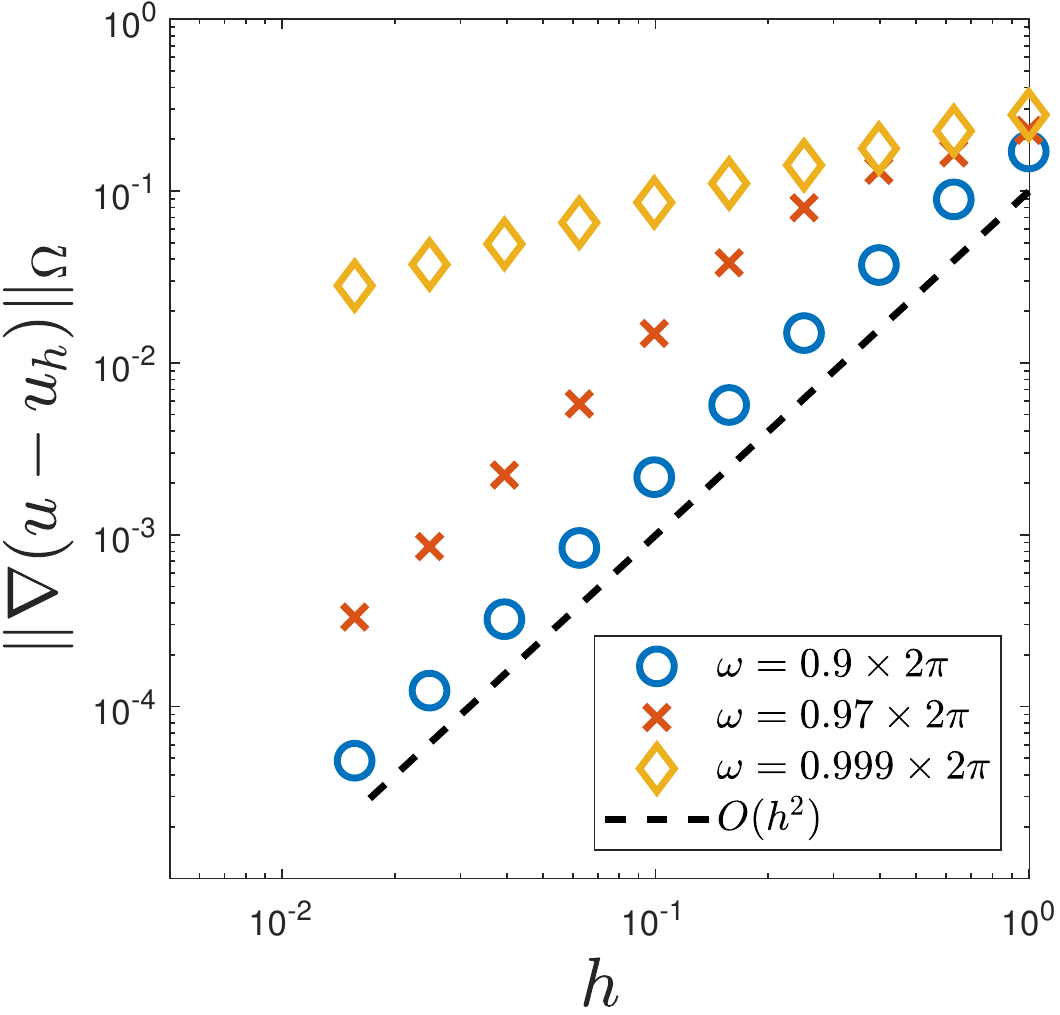}	
		\caption{Error in $H^1$ semi-norm}
		\label{fig:convergence-nosplit-b}
	\end{subfigure}
	\begin{subfigure}[t]{0.32\linewidth}\centering
		\includegraphics[width=1\linewidth]{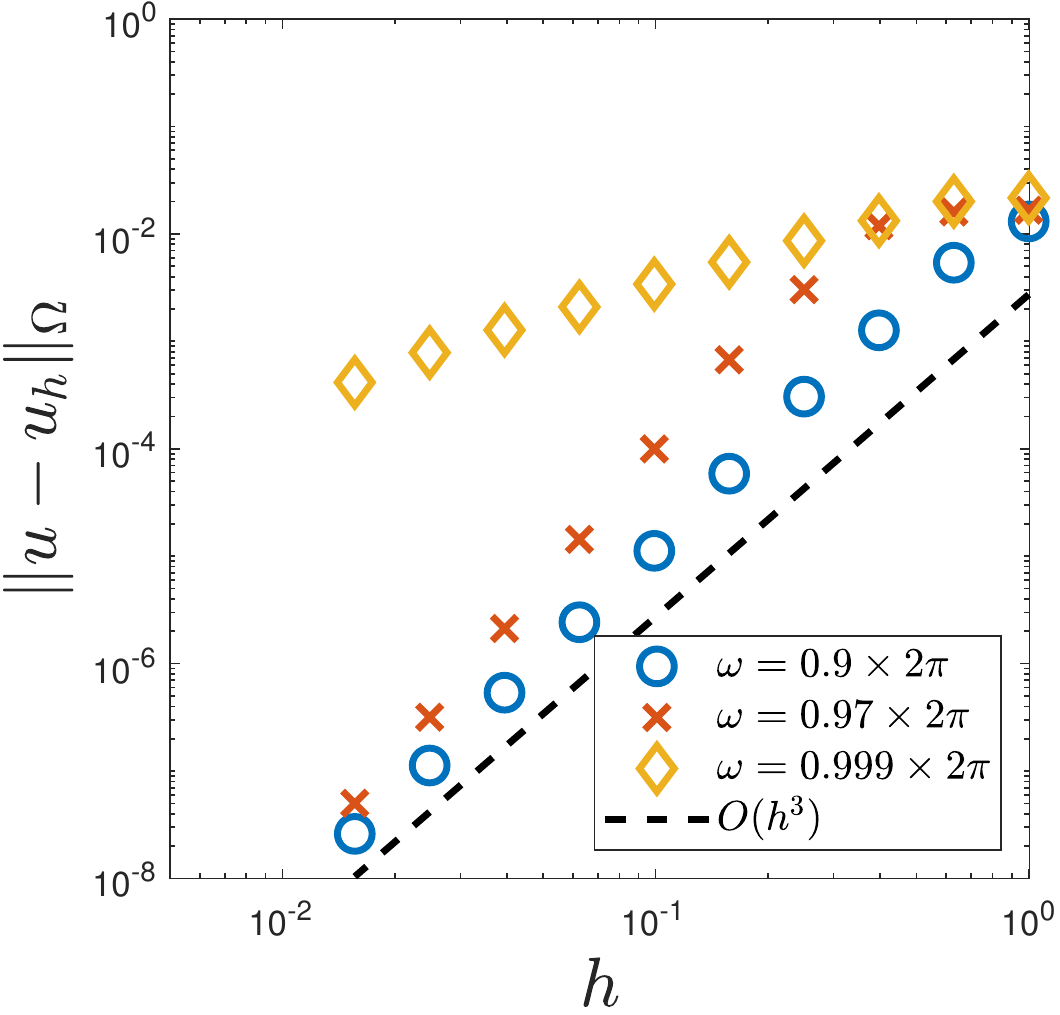}	
		\caption{Error in $L^2$ norm}
		\label{fig:convergence-nosplit-c}
	\end{subfigure}
\caption{
Convergence when applying the method to $C^1$ spline approximation spaces on a cut tensor product mesh for the model problem with various opening angles $\omegac$.
(a) Numerical solution with opening angle $\omegac=0.97\times 2\pi$. Note that the effect of the parametric mapping $F_\gamma$ is that the mesh, viewed in physcial coordinates, is graded towards the singularity. 
(b)--(c) Convergence in $H^1(\Omega)$ semi-norm and in $L^2(\Omega)$ norm. In the more extreme choices of opening $\omegac$ we do not, at least not initially, have optimal order convergence. We attribute this effect to unwanted coupling over the slit void caused by the cut approximation space being independent of the geometry and worsened by the relatively large support of $C^1$ spline basis functions.
} 
\label{fig:convergence-nosplit}
\end{figure}

\paragraph{Fix for Removing Unwanted Coupling.}
As a simple fix for removing the undesired coupling over the slit we propose the following approach:
Basis functions whose support consists of two disjoint parts, one above and one below the slit, are associated with two separate degrees of freedom, one connected to elements above the slit and one connected to elements below the slit. This is illustrated in Figure~\ref{fig:basis-split-a}. Note that this fix will not effect the regularity of the numerical solution. However, in the vicinity of the corner $c$ there will still exist some unwanted coupling as illustrated in Figure~\ref{fig:basis-split-b}.

\begin{figure}
\centering	
	\begin{subfigure}[t]{0.48\linewidth}\centering
		\includegraphics[width=0.87\linewidth]{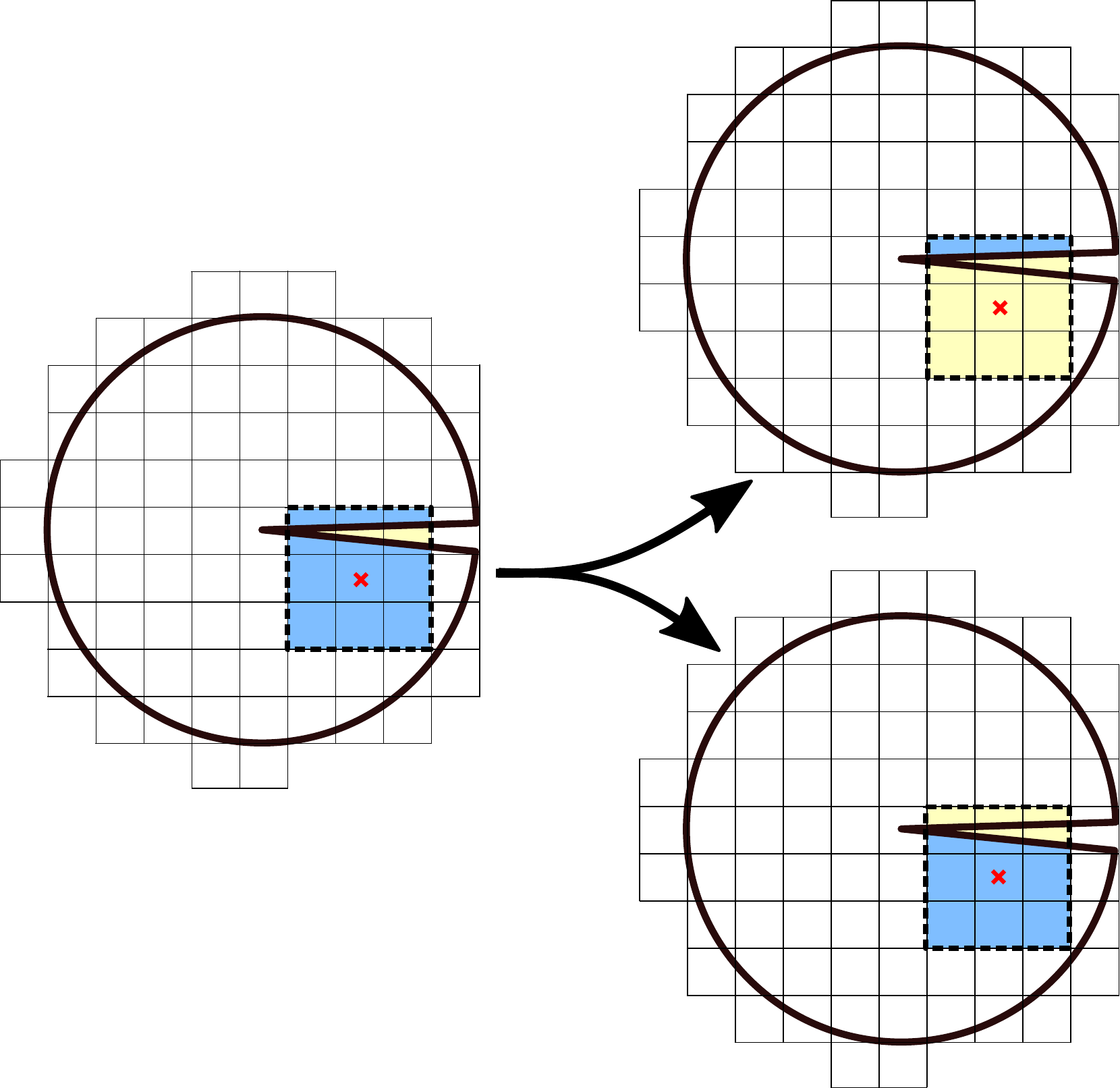}	
		\caption{Split of disjoint $\mathrm{supp}(\varphi_i) \cap \Omega$}
		\label{fig:basis-split-a}
	\end{subfigure}
	\begin{subfigure}[t]{0.48\linewidth}\centering
		\includegraphics[width=0.8\linewidth,trim=0 -125 0 0, clip]{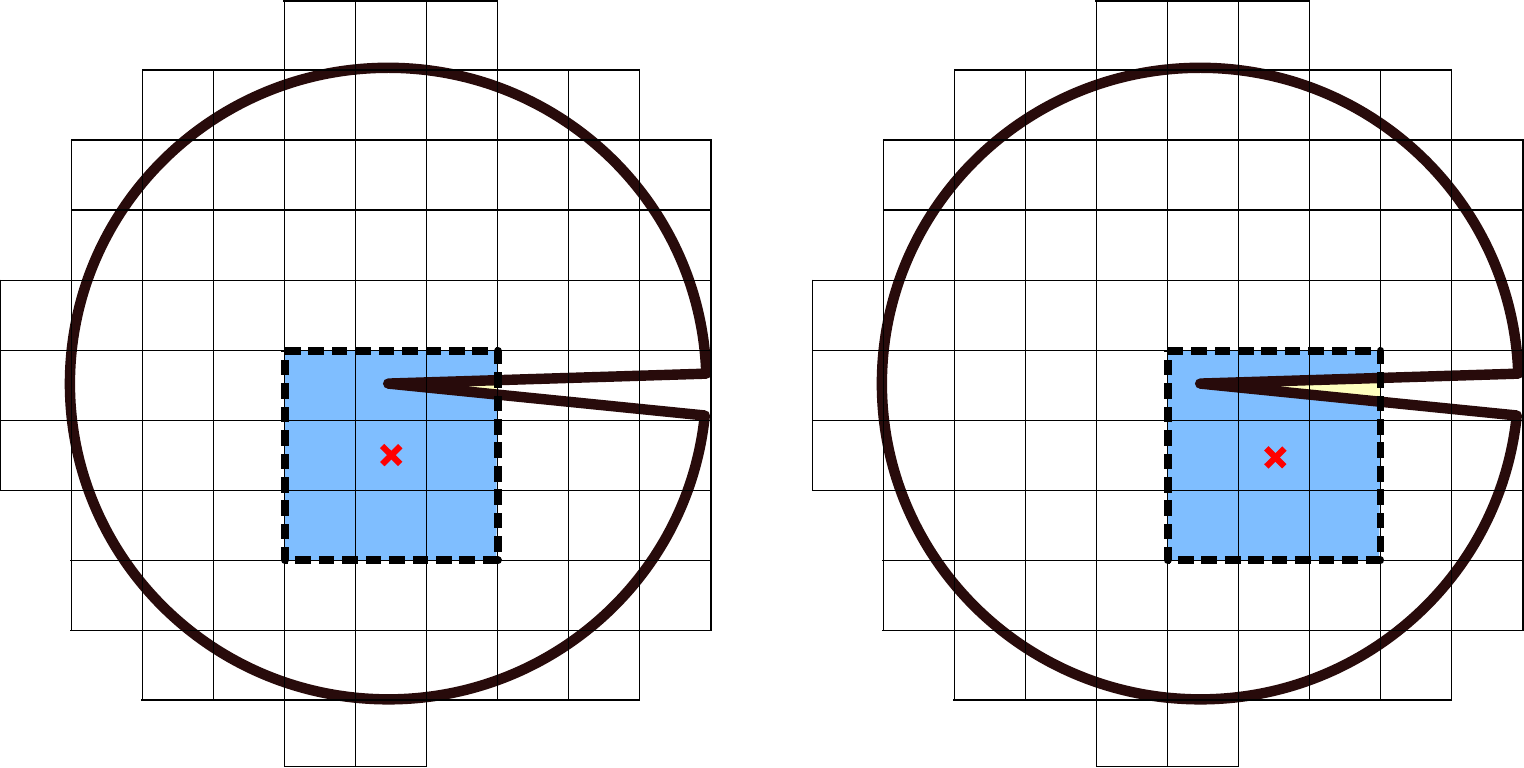}	
		\caption{Pathwise-connected $\mathrm{supp}(\varphi_i) \cap \Omega$}
		\label{fig:basis-split-b}
	\end{subfigure}
\caption{Fix for unwanted coupling.
(a) Illustration of how a basis function causing unwanted coupling over the slit where $\mathrm{supp}(\varphi_i) \cap \Omega$ consists of two disjoint parts may be viewed as two separate basis functions, associated with the part above respectively below the slit.
(b) Support of two basis functions causing unwanted coupling over the slit that we are unable to fix since the intersection between the support of each basis function and the domain is pathwise-connected.
} 
	\label{fig:basis-split}
\end{figure}

We numerically assess the effect of this fix in Figure~\ref{fig:convergence-split} where we in (a)--(b) note that we with again recover the optimal order convergence implied by our estimates also for opening angles $\omegac$ close to $2\pi$ when using a cut $C^1$ spline approximation space. In Figures~\ref{fig:convergence-split-c}--d we note that the effect of the fix is more pronounced the larger the opening angle. We should however remark that this fix is only necessary in more extreme cases of opening angles or when using a quite large mesh size.

\begin{figure}
\centering	
	\begin{subfigure}[t]{0.32\linewidth}\centering
		\includegraphics[width=1\linewidth]{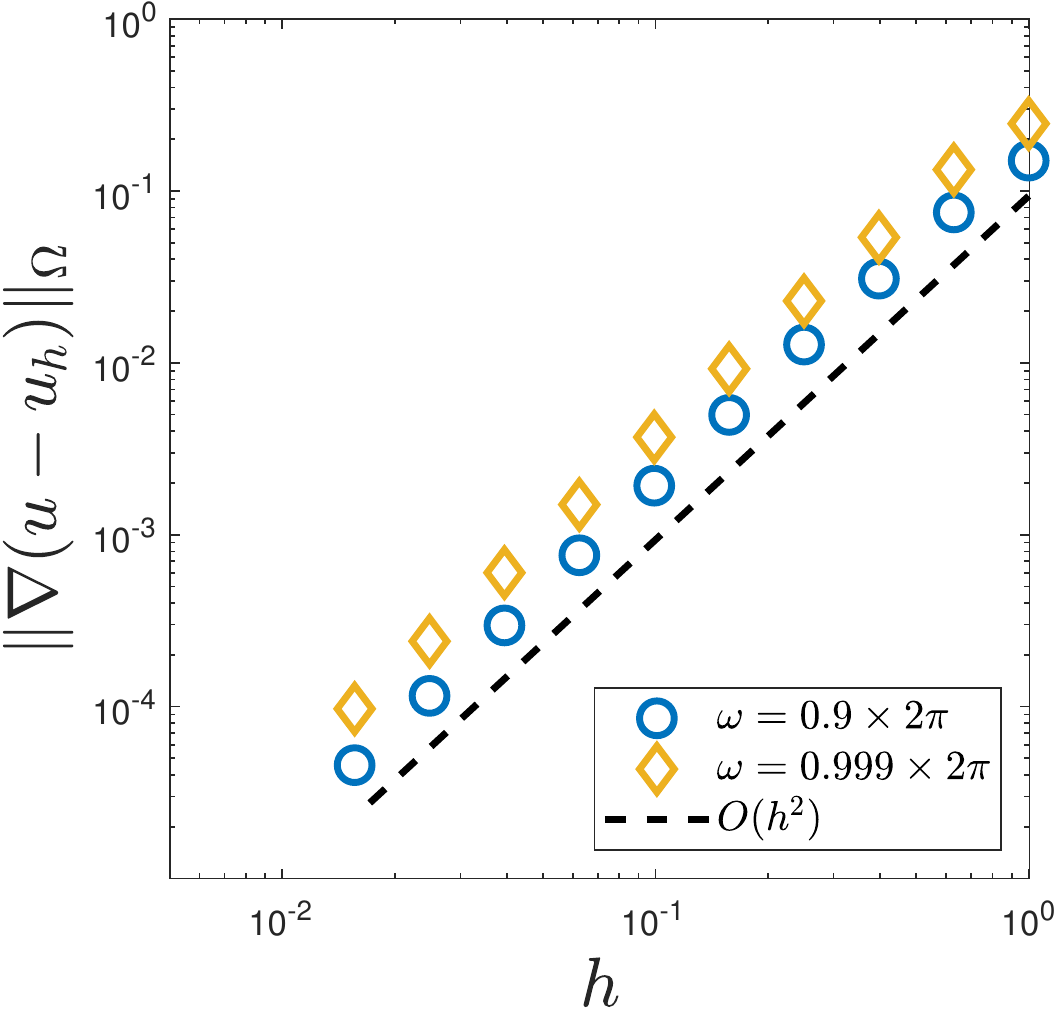}	
		\caption{Error in $H^1$ semi-norm}
		\label{fig:convergence-split-a}
	\end{subfigure}
	\
	\begin{subfigure}[t]{0.32\linewidth}\centering
		\includegraphics[width=1\linewidth]{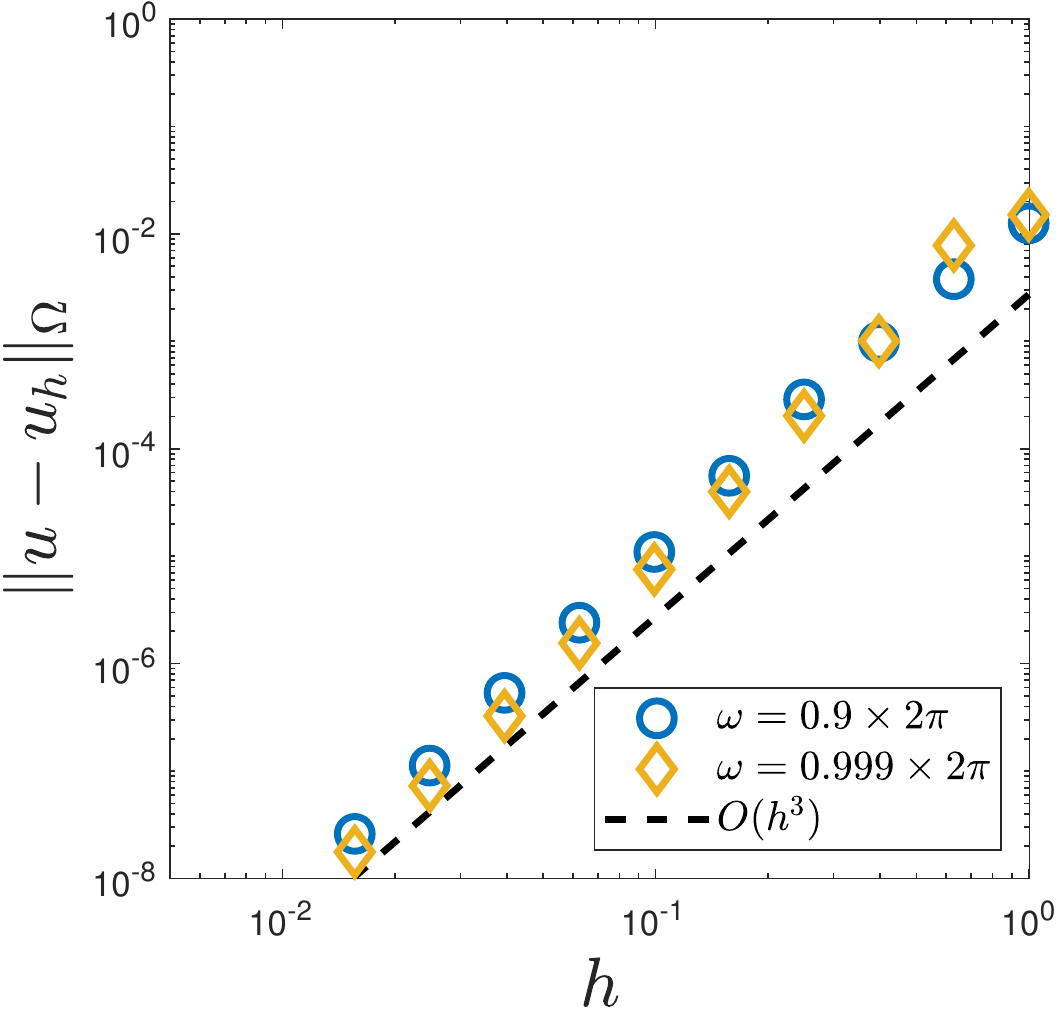}	
		\caption{Error in $L^2$ norm}
		\label{fig:convergence-split-b}
	\end{subfigure}

\vspace{1ex}
		\begin{subfigure}[t]{0.32\linewidth}\centering
		\includegraphics[width=1\linewidth]{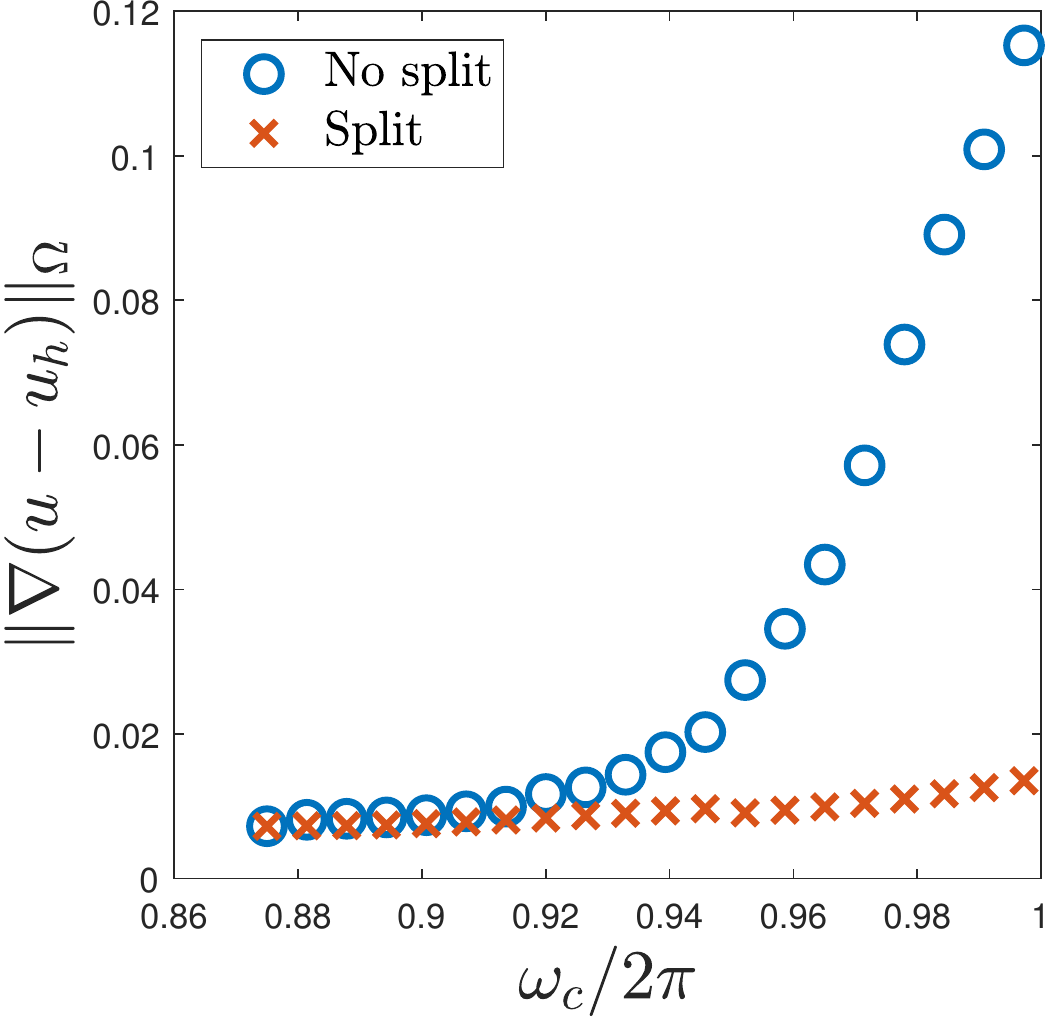}	
		\caption{Error in $H^1$ semi-norm}
		\label{fig:convergence-split-c}
	\end{subfigure}
	\
	\begin{subfigure}[t]{0.32\linewidth}\centering
		\includegraphics[width=0.96\linewidth]{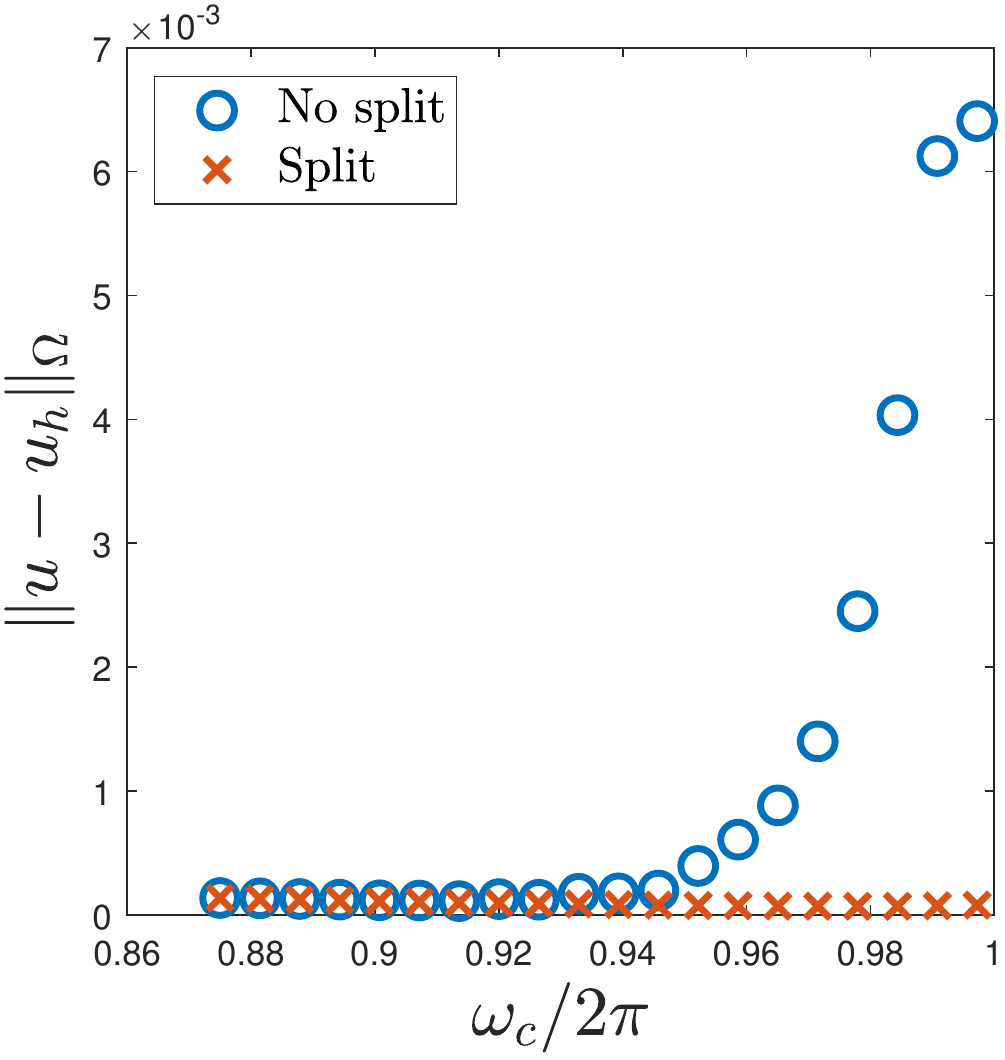}	
		\caption{Error in $L^2$ norm}
		\label{fig:convergence-split-d}
	\end{subfigure}
\caption{
Effects of applying the fix to remove unwanted coupling over the slit to the model problem when using cut $C^1$ spline approximation spaces.
(a)--(b) Convergence in $H^1(\Omega)$ semi-norm and in $L^2(\Omega)$ norm for two different opening angles $\omegac$ with the fix applied. Note that we again recover optimal order convergence, cf. Figure~\ref{fig:convergence-nosplit}.
(c)--(d) Parameter study of how the opening angle $\omegac$ effects the $L^2(\Omega)$ norm and $H^1(\Omega)$ semi-norm errors using a fixed mesh size $h=0.2$.
} 
\label{fig:convergence-split}
\end{figure}

\paragraph{Stability with Respect to Mesh Position.}
In the previous examples we have positioned the mesh such that the nonconvex corner $c$ is located in the center of an element when using $C^1$ spline approximation spaces.
To study how the mesh position in relation to the nonconvex corner $c$ effects the performance of the method we in Figure~\ref{fig:corner-pos} consider 400 random positions of the mesh in a $C^1$ spline approximation space with fixed mesh size $h=0.1$. We note that the method is actually quite insensitive with respect to the mesh position with a standard deviation of the error relative its mean of $0.02$ in $H^1(\Omega)$ semi-norm and of $0.04$ in $L^2(\Omega)$ norm. 

\begin{figure}
	\centering
	\begin{subfigure}{0.32\linewidth}\centering
		\includegraphics[width=1\linewidth]{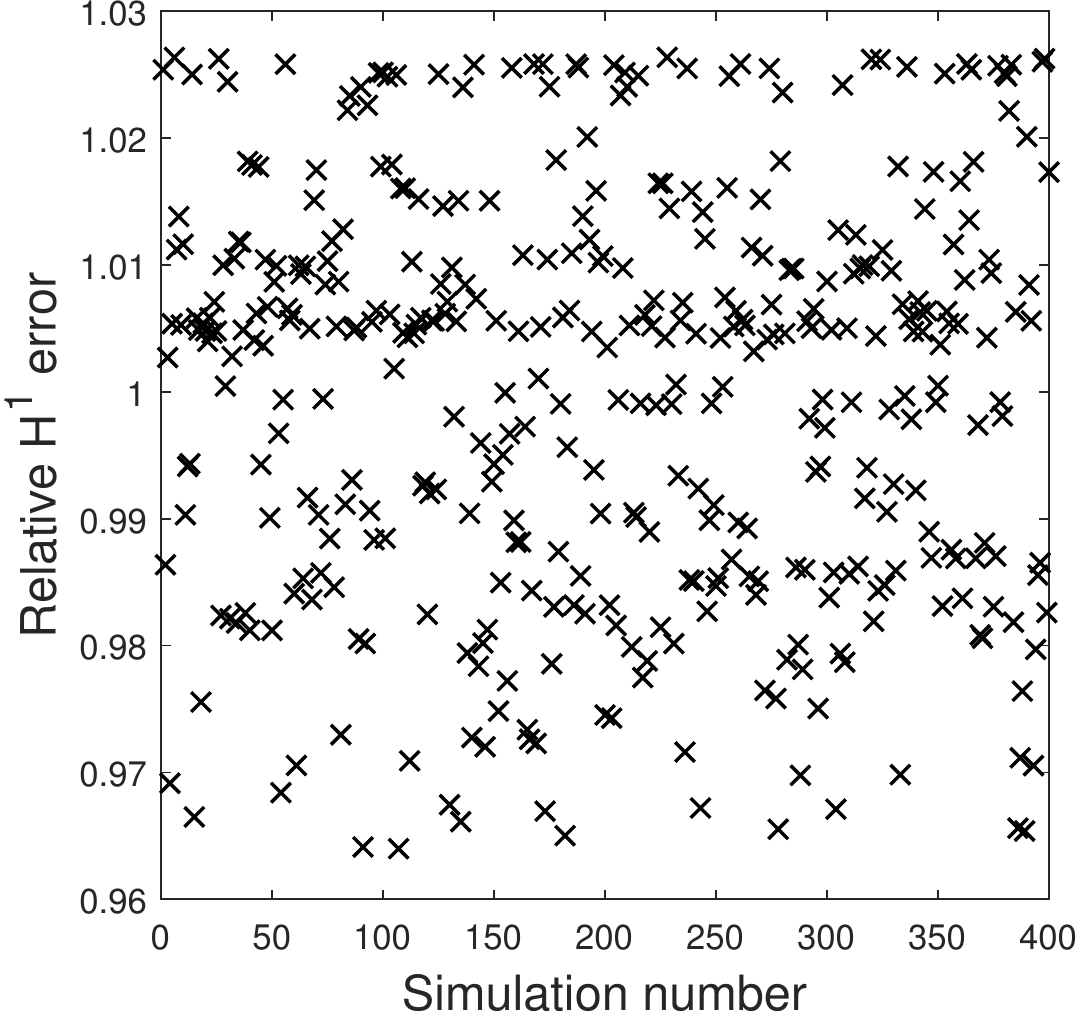}
		\subcaption{Error in $H^1$ semi-norm}
		\label{fig:corner-pos-a}
	\end{subfigure}
	\
	\begin{subfigure}{0.32\linewidth}\centering
		\includegraphics[width=1\linewidth]{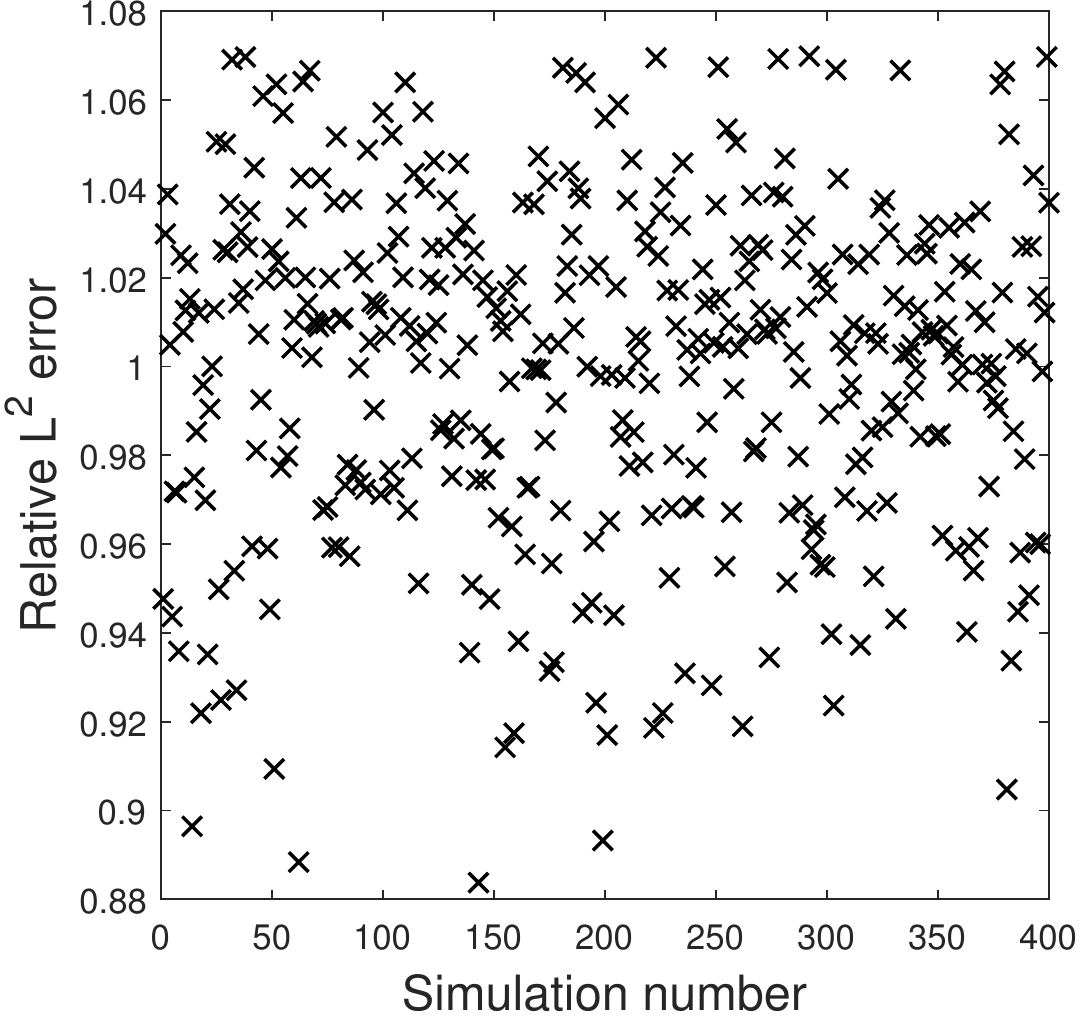}
		\subcaption{Error in $L^2$ norm}
		\label{fig:corner-pos-b}
	\end{subfigure}
	\centering
\caption{
Study of how the mesh position with respect to the location of the nonconvex corner effects the error when using a $C^1$ spline approximation space with a mesh size $h=0.1$ for the model problem with opening angle $\omega_c = 0.75\times 2\pi$. In (a) and (b) errors in $H^1(\Omega)$ semi-norm respectively $L^2(\Omega)$ norm relative to the mean error is presented for 400 uniformly distributed random positions of the mesh. Note that most errors lie within $\pm 4\,\%$ in (a) respectively $\pm 10\,\%$ in (b).
}
\label{fig:corner-pos}
\end{figure}

\subsection{More Examples}

\paragraph{Multiple Nonconvex Corners.}
In Figure~\ref{fig:fesol_multipatch} we present a numerical example on a domain featuring two nonconvex corners.
This is solved using the multipatch approach outlined in Section~\ref{section:multipatch}. The domain is partitioned into three patches where each patch is equipped with its own approximation space and the appropriately chosen map. We note that the finite element solution and gradient magnitude seem to flow nicely over the internal interfaces.

\begin{figure}
	\centering
		\begin{subfigure}[t]{0.32\linewidth}\centering
		\includegraphics[width=0.9\linewidth]{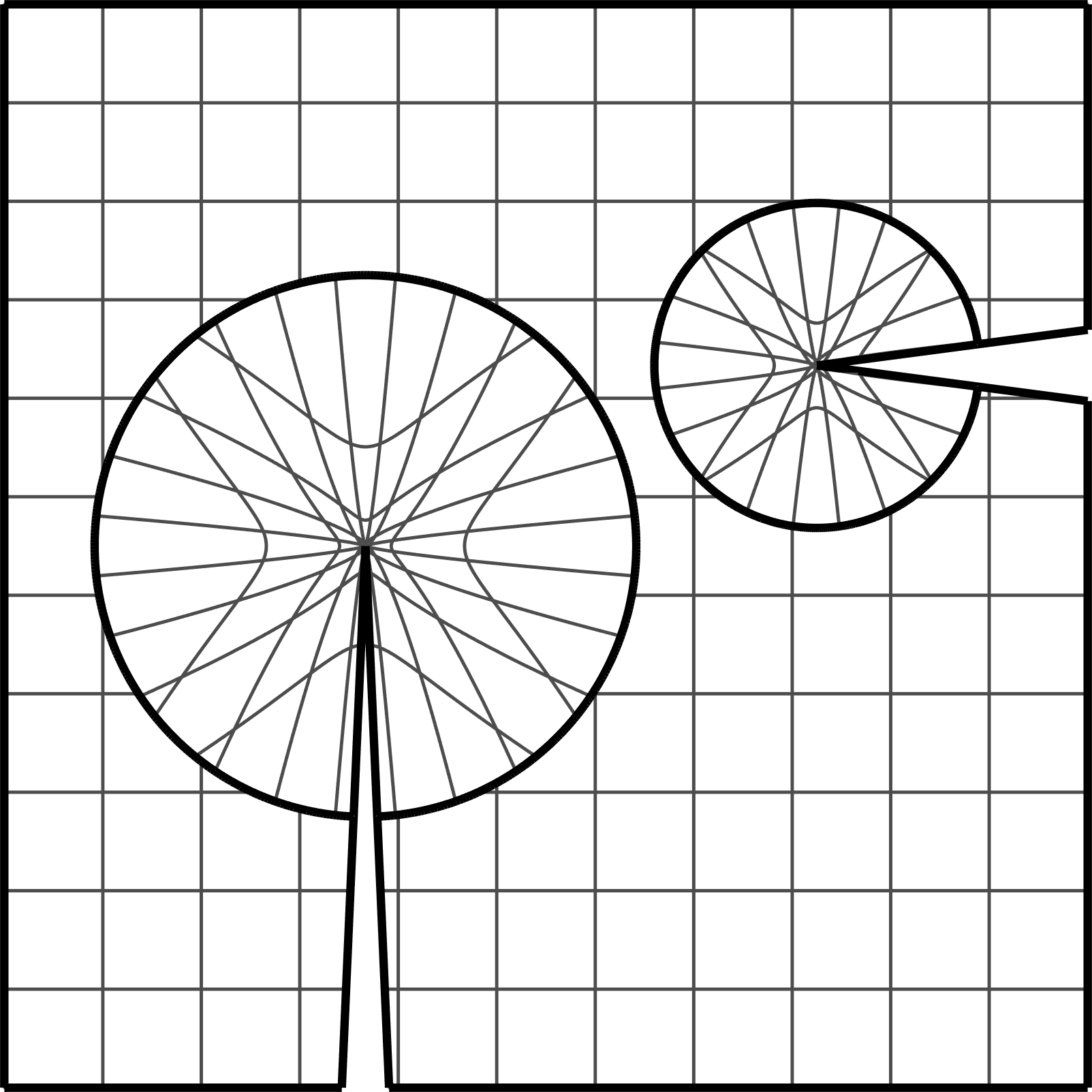}
		\subcaption{Multipatch mesh in the physical domain}
	\end{subfigure}
\
	\begin{subfigure}[t]{0.32\linewidth}\centering
		\includegraphics[width=0.9\linewidth]{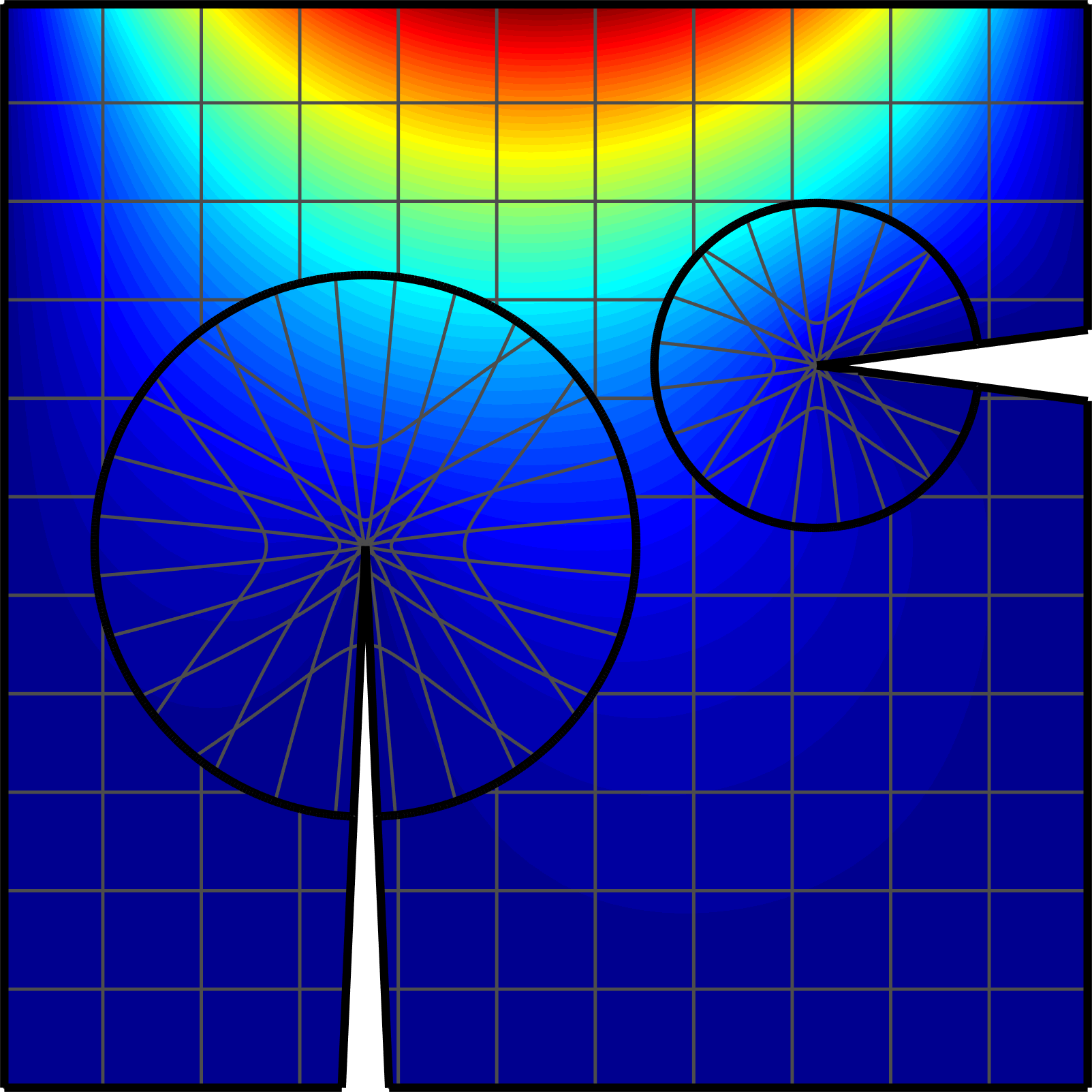}
		%	\label{fig:fitted-refmesh}
		\subcaption{Numerical solution}
	\end{subfigure}
\
	\begin{subfigure}[t]{0.32\linewidth}\centering
		\includegraphics[width=0.9\linewidth]{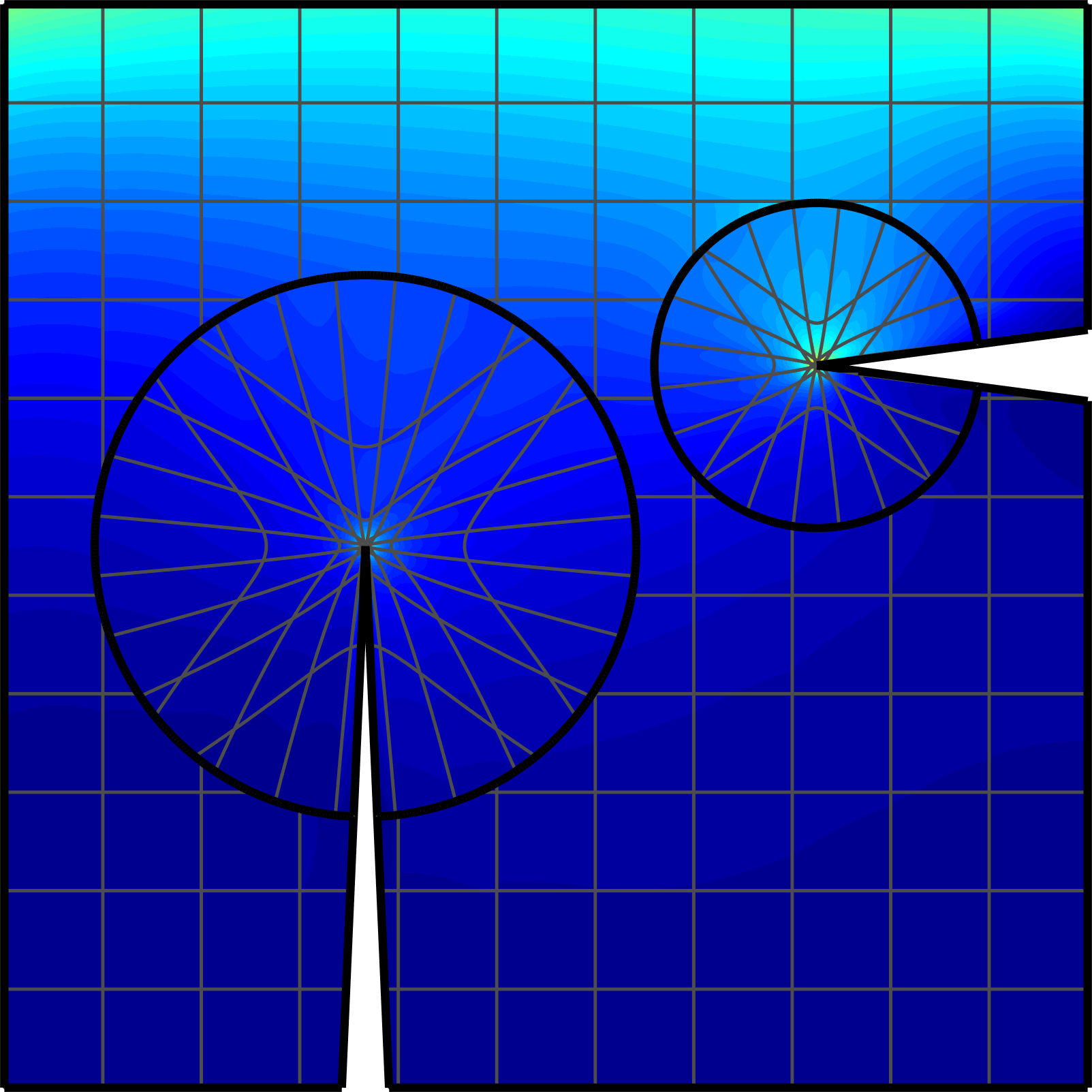}
		\subcaption{Gradient magnitude}
	\end{subfigure}
\caption{Example with two nonconvex corners with $g=x(1-x)$ on the top edge and $g=0$ on the remaining boundary.}
\label{fig:fesol_multipatch}
\end{figure}

\paragraph{Curved Surface.}
Via the parametric map the method naturally handles problems on a curved surface $\Omega_{\IR^3}\subset \IR^3$ by the composite map $F=F_* \circ F_\gamma$ where $F_* : \Omega \to \Omega_{\IR^3}$. As an illustration we present an example solution on a curved surface in Figure~\ref{fig:curvdSrf}.

\begin{figure}
	\centering
	\begin{subfigure}[t]{0.32\linewidth}\centering
		\includegraphics[width=0.9\linewidth]{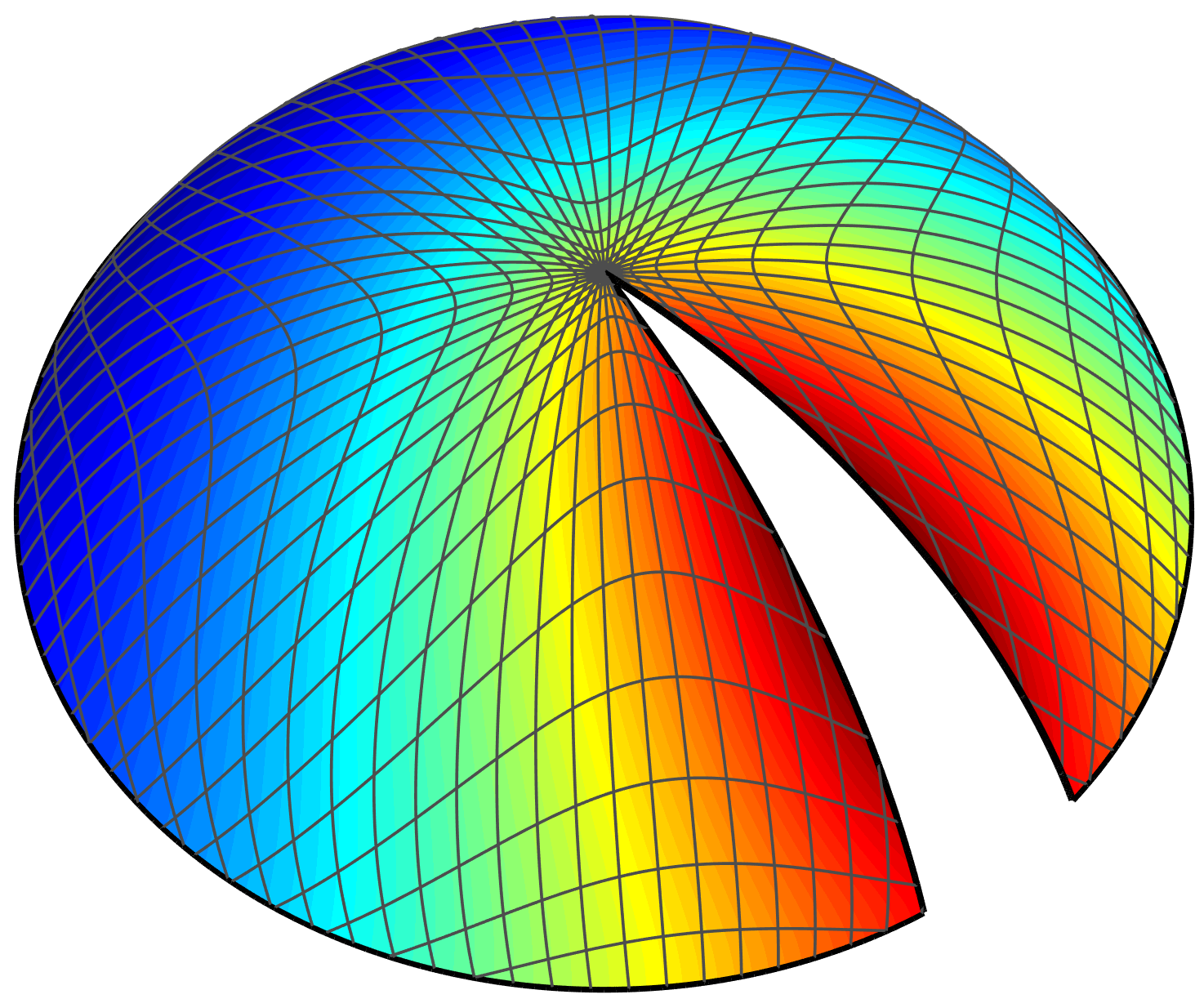}
		\subcaption{Numerical solution}
	\end{subfigure}
	\quad
	\begin{subfigure}[t]{0.32\linewidth}\centering
		\includegraphics[width=0.9\linewidth]{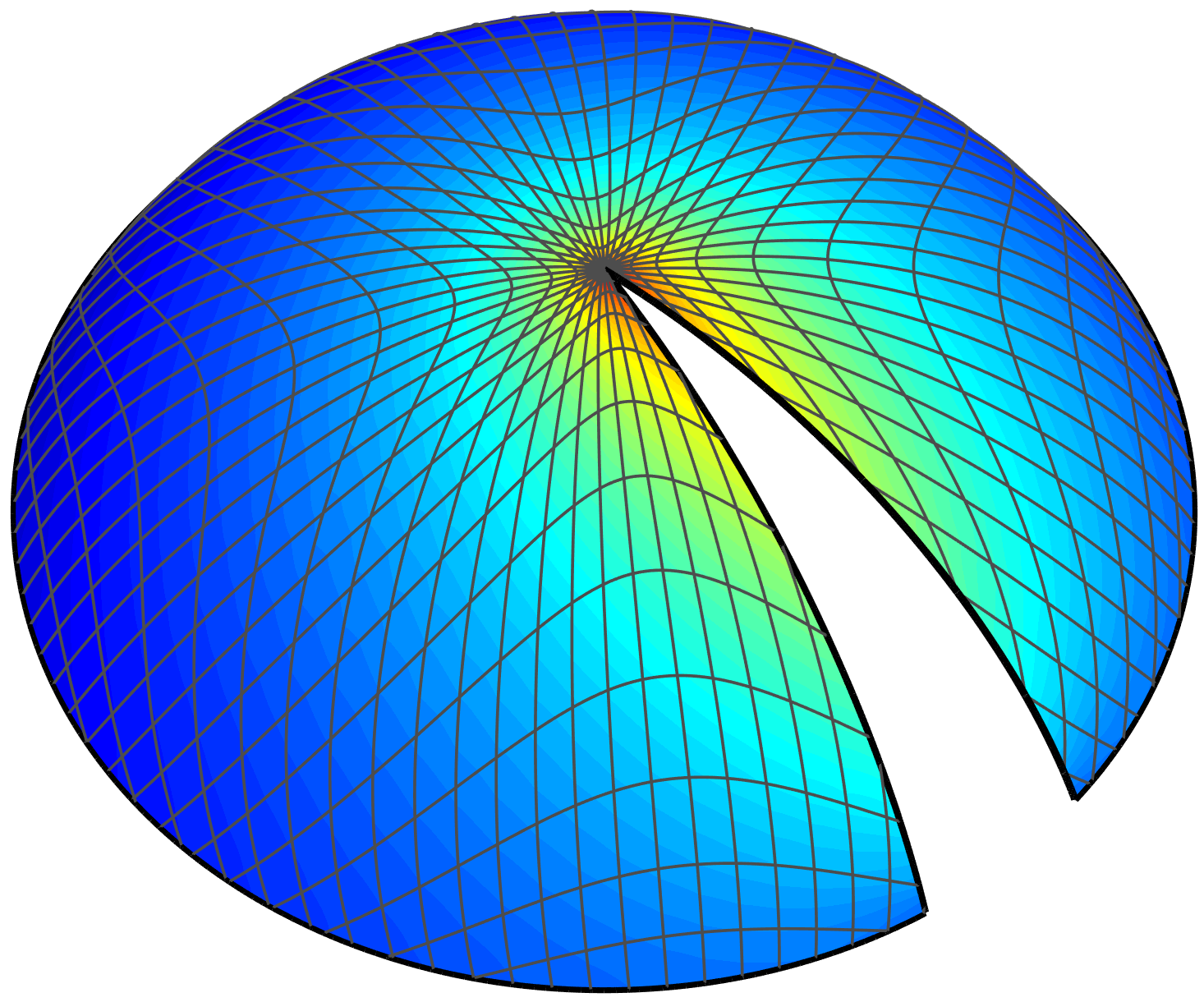}
		\subcaption{Magnitude of the gradient}
	\end{subfigure}
	\caption{
	Numerical solution and gradient magnitude to problem on a curved surface.
	In contrast to the rest of the paper we in this specific example use Neumann boundary conditions.
	}
	\label{fig:curvdSrf}
\end{figure}

\section{Conclusions} \label{section:conclusions}
We have developed a new parametric higher-order cut finite element method for elliptic boundary value problems with corner singularities. It has the following notable features:
\begin{itemize}
\item The method is based on a radial map that suitably grades the mesh towards the singularity. This map can be chosen without knowing the exact opening angle $\omegac < 2\pi$.

\item Numerical instabilities due to unbounded derivatives of the map near the corner are avoided by formulating the method in a reference domain.

\item The method is proven to be stable and to be optimal order convergent in energy and $L^2(\Omega)$ norms.

\item Multiple nonconvex corners are handled by using a previously developed multipatch framework such that each corner can be dealt with individually. 

\item Unwanted coupling over voids induced by the combination of extreme opening angles $\omega_c$ and cut approximation spaces is remedied by a proposed fix that restores the initial optimal order convergence to a large extent.

\end{itemize}

\paragraph{Acknowledgements.} This research was supported in part by the Swedish Foundation
for Strategic Research Grant No. AM13-0029, the Swedish Research Council Grants Nos.
2013-4708, 2017-03911 and the Swedish Research Programme Essence.

\clearpage

\bibliographystyle{habbrv}
{
\footnotesize
\bibliography{weighted-sobolev-refs}

\begin{thebibliography}{10}
\expandafter\ifx\csname url\endcsname\relax
  \def\url#1{\texttt{#1}}\fi
\expandafter\ifx\csname doi\endcsname\relax
  \def\doi#1{\burlalt{doi:#1}{http://dx.doi.org/#1}}\fi
\expandafter\ifx\csname urlprefix\endcsname\relax\def\urlprefix{URL }\fi
\expandafter\ifx\csname href\endcsname\relax
  \def\href#1#2{#2}\fi
\expandafter\ifx\csname burlalt\endcsname\relax
  \def\burlalt#1#2{\href{#2}{#1}}\fi

\bibitem{AS1997}
M.~Ainsworth and B.~Senior.
\newblock Aspects of an adaptive {$hp$}-finite element method: adaptive
  strategy, conforming approximation and efficient solvers.
\newblock {\em Comput. Methods Appl. Mech. Engrg.}, 150(1-4):65--87, 1997.
\newblock \doi{10.1016/S0045-7825(97)00101-1}.
\newblock Symposium on Advances in Computational Mechanics, Vol. 2 (Austin, TX,
  1997).

\bibitem{BanRan2003}
W.~Bangerth and R.~Rannacher.
\newblock {\em Adaptive finite element methods for differential equations}.
\newblock Lectures in Mathematics ETH Z\"{u}rich. Birkh\"{a}user Verlag, Basel,
  2003.
\newblock \doi{10.1007/978-3-0348-7605-6}.

\bibitem{BelGraVen2009}
T.~Belytschko, R.~Gracie, and G.~Ventura.
\newblock A review of extended/generalized finite element methods for material
  modeling.
\newblock {\em Model. Simul. Mater. Sci. Eng.}, 17(4):043001, 2009.
\newblock \doi{10.1088/0965-0393/17/4/043001}.

\bibitem{Burman2010}
E.~Burman.
\newblock Ghost penalty.
\newblock {\em C. R. Math. Acad. Sci. Paris}, 348(21-22):1217--1220, 2010.
\newblock \doi{10.1016/j.crma.2010.10.006}.

\bibitem{BuClHaLaMa15}
E.~Burman, S.~Claus, P.~Hansbo, M.~G. Larson, and A.~Massing.
\newblock Cut{FEM}: discretizing geometry and partial differential equations.
\newblock {\em Internat. J. Numer. Methods Engrg.}, 104(7):472--501, 2015.
\newblock \doi{10.1002/nme.4823}.

\bibitem{BuHa2012}
E.~Burman and P.~Hansbo.
\newblock Fictitious domain finite element methods using cut elements: {II}.
  {A} stabilized {N}itsche method.
\newblock {\em Appl. Numer. Math.}, 62(4):328--341, 2012.
\newblock \doi{10.1016/j.apnum.2011.01.008}.

\bibitem{IGABook}
J.~A. Cottrell, T.~J.~R. Hughes, and Y.~Bazilevs.
\newblock {\em Isogeometric Analysis: Toward Integration of CAD and FEA}.
\newblock Wiley Publishing, 1st edition, 2009.
\newblock \doi{10.1002/9780470749081}.

\bibitem{ElfLarLar18a}
D.~Elfverson, M.~G. Larson, and K.~Larsson.
\newblock Cut{IGA} with basis function removal.
\newblock {\em Adv. Model. Simul. Eng. Sci.}, 5(6):1--19, 2018.
\newblock \doi{10.1186/s40323-018-0099-2}.

\bibitem{ElfLarLar18b}
D.~{Elfverson}, M.~G. {Larson}, and K.~{Larsson}.
\newblock A new least squares stabilized {N}itsche method for cut isogeometric
  analysis.
\newblock {\em Comput. Methods Appl. Mech. Engrg.}, 349:1--16, 2019.
\newblock \doi{10.1016/j.cma.2019.02.011}.

\bibitem{FriBel2010}
T.-P. Fries and T.~Belytschko.
\newblock The extended/generalized finite element method: an overview of the
  method and its applications.
\newblock {\em Internat. J. Numer. Methods Engrg.}, 84(3):253--304, 2010.
\newblock \doi{10.1002/nme.2914}.

\bibitem{HanHanLar2003}
A.~Hansbo, P.~Hansbo, and M.~G. Larson.
\newblock A finite element method on composite grids based on {N}itsche's
  method.
\newblock {\em M2AN Math. Model. Numer. Anal.}, 37(3):495--514, 2003.
\newblock \doi{10.1051/m2an:2003039}.

\bibitem{IGA}
T.~J.~R. Hughes, J.~A. Cottrell, and Y.~Bazilevs.
\newblock Isogeometric analysis: {CAD}, finite elements, {NURBS}, exact
  geometry and mesh refinement.
\newblock {\em Comput. Methods Appl. Mech. Engrg.}, 194(39-41):4135--4195,
  2005.
\newblock \doi{10.1016/j.cma.2004.10.008}.

\bibitem{JeOhKaKi2013}
J.~W. Jeong, H.-S. Oh, S.~Kang, and H.~Kim.
\newblock Mapping techniques for isogeometric analysis of elliptic boundary
  value problems containing singularities.
\newblock {\em Comput. Methods Appl. Mech. Engrg.}, 254:334--352, 2013.
\newblock \doi{10.1016/j.cma.2012.09.009}.

\bibitem{MR3627181}
Y.~Jin, O.~A. Gonz\'{a}lez-Estrada, O.~Pierard, and S.~P.~A. Bordas.
\newblock Error-controlled adaptive extended finite element method for 3{D}
  linear elastic crack propagation.
\newblock {\em Comput. Methods Appl. Mech. Engrg.}, 318:319--348, 2017.
\newblock \doi{10.1016/j.cma.2016.12.016}.

\bibitem{JoLaLa2017}
T.~Jonsson, M.~G. Larson, and K.~Larsson.
\newblock Cut finite element methods for elliptic problems on multipatch
  parametric surfaces.
\newblock {\em Comput. Methods Appl. Mech. Engrg.}, 324:366--394, 2017.
\newblock \doi{10.1016/j.cma.2017.06.018}.

\bibitem{MR0226187}
V.~A. Kondrat'ev.
\newblock Boundary value problems for elliptic equations in domains with
  conical or angular points.
\newblock {\em Trudy Moskov. Mat. Ob\v{s}\v{c}.}, 16:209--292, 1967.

\bibitem{LiLu2000}
Z.~C. Li and T.~T. Lu.
\newblock Singularities and treatments of elliptic boundary value problems.
\newblock {\em Math. Comput. Modelling}, 31(8-9):97--145, 2000.
\newblock \doi{10.1016/S0895-7177(00)00062-5}.

\bibitem{MaLaLoRo2013a}
A.~Massing, M.~G. Larson, A.~Logg, and M.~E. Rognes.
\newblock A stabilized {N}itsche fictitious domain method for the {S}tokes
  problem.
\newblock {\em J. Sci. Comput.}, 61(3):604--628, 2014.
\newblock \doi{10.1007/s10915-014-9838-9}.

\bibitem{MR3372009}
V.~P. Nguyen, C.~Anitescu, S.~P.~A. Bordas, and T.~Rabczuk.
\newblock Isogeometric analysis: an overview and computer implementation
  aspects.
\newblock {\em Math. Comput. Simulation}, 117:89--116, 2015.
\newblock \doi{10.1016/j.matcom.2015.05.008}.

\bibitem{OhKiJe2014}
H.-S. Oh, H.~Kim, and J.~W. Jeong.
\newblock Enriched isogeometric analysis of elliptic boundary value problems in
  domains with cracks and/or corners.
\newblock {\em Internat. J. Numer. Methods Engrg.}, 97(3):149--180, 2014.
\newblock \doi{10.1002/nme.4580}.

\end{thebibliography}
}

%\vfill
\bigskip
\bigskip
{
\footnotesize{
\begin{samepage}

\bigskip
\bigskip
\noindent
{\bf Authors' addresses:}

\smallskip
\noindent
Tobias Jonsson,  \quad \hfill \addressumushort\\
{\tt tobias.jonsson@umu.se}

\smallskip
\noindent
Mats G. Larson,  \quad \hfill \addressumushort\\
{\tt mats.larson@umu.se}

\smallskip
\noindent
Karl Larsson, \quad \hfill \addressumushort\\
{\tt karl.larsson@umu.se}
\end{samepage}
}

\clearpage
\appendix
\normalsize

\section{Riemannian Calculus Approach} \label{appendix:riemann}

We first change coordinates from Euclidean to polar and then we change 
from polar to weighted polar.

\subsection{Polar Coordinates}

Consider the mapping $F_P : B_P \rightarrow B \subset \IR^2$ where
$\IR^2$ is equipped with the Euclidean inner product and 
\begin{equation}
F_P(r,\theta) = 
\left[\begin{matrix}  
r\cos \theta
\\
r\sin \theta
\end{matrix}
\right]
\end{equation}
with the partial derivatives 
\begin{equation}
\partial_r F_P = \left[\begin{matrix}  
\cos \theta
\\
\sin \theta
\end{matrix}
\right],
\qquad
\partial_\theta F_P = \left[\begin{matrix}  
-r\sin \theta
\\
r\cos \theta
\end{matrix}
\right]
\end{equation}
The metric tensor $G_P$ is defined $(G_P)_{ij} = \partial_i F_P \cdot \partial_j F_P$
which after simplification takes the form
%where we recall that by using Euclidean inner product it takes the form
\begin{equation}
G_P = \left[\begin{matrix}  
1 & 0
\\
0 & r^2
\end{matrix}
\right]
\end{equation}
with inverse
\begin{equation}
G_P^{-1} 
= \left[\begin{matrix}  
1 & 0
\\
0 & r^{-2}
\end{matrix}
\right]
\end{equation}
and determinant 
\begin{equation}
|G| = r^{2}
\end{equation}

\paragraph{Gradient.}
Let $v: B \rightarrow \IR$ and let $v_P: B_P \rightarrow \IR$ be the pullback 
\begin{equation}
v_P = v \circ F_P
\end{equation}
The gradient satisfies
\begin{equation}
\nabla v = (D_PF_P) \cdot  (\nabla v )_P, \qquad (\nabla v)_P = G_P^{-1} \nabla_P v_P
\end{equation} 
which gives 
\begin{equation}
(\nabla v)_P = 
 \left[
 \begin{matrix}  
1 & 0
\\
0 & r^{-2}
\end{matrix}
\right]
\; 
 \left[\begin{matrix}  
\partial_r v_P
\\
\partial_\theta v_P 
\end{matrix}
\right]
\end{equation}
and 
\begin{align}
\nabla v &= 
 \left[\begin{matrix}  
\cos \theta & -r \sin \theta 
\\
\sin \theta & r \cos \theta
\end{matrix}
\right]
 \left[
 \begin{matrix}  
1 & 0
\\
0 & r^{-2}
\end{matrix}
\right]
\; 
 \left[\begin{matrix}  
\partial_r v_P
\\
\partial_\theta v_P 
\end{matrix}
\right]
\\
&= 
 \left[\begin{matrix}  
\cos \theta 
\\
\sin \theta
\end{matrix}
\right]\partial_r v_P 
+ 
 r^{-1} \left[\begin{matrix}  
 -\sin \theta 
\\
 \cos \theta
\end{matrix}
\right]
 \partial_\theta v_P 
\end{align}

\paragraph{Bilinear Form.}
The bilinear form associated with the Laplacian transforms as follows
\begin{align}
\int_B \nabla v \cdot \nabla w d\mu 
&=
\int_{B_P} (\nabla v)_P \cdot G_P  \cdot (\nabla w)_P  d\mu_P
\\
&=
\int_{B_P} (\nabla_P v_P ) \cdot G^{-1}_P  \cdot (\nabla_P w_P)  |G_P|^{1/2} d r d\theta
\\
&= 
\int_{B_P} \left( \partial_r v_P \partial_r w_P + r^{-2} \partial_\theta v_P \partial_\theta w_P \right)   r d r d\theta
\end{align} 
where $d\mu = dx dy$ is the standard measure on $\IR^2$ and $d\mu_P$ is the pullback 
measure.

\paragraph{Consistency Term.}
Correspondingly the consistency term transforms as follows
\begin{align}
\int_{\partial B} (n\cdot\nabla v) w \, ds
&=
\int_{\partial B_P} ((n)_P \cdot G_P \cdot (\nabla v)_P) w_P \, (ds)_P
\end{align}
where $ds$ is the standard line measure on $\IR^2$ and $(ds)_P$ is the pullback measure.

\subsection{Radial Mapping in Polar Coordinates} 

Consider the mapping from $B_{\hatP} \rightarrow B_P$ where $B_P$ is a disc 
centered at the origin such that 
\begin{equation}
F_{\hatP}: 
\left[
\begin{matrix}
\hatr
\\
\hattheta
\end{matrix}
\right]
\mapsto 
\left[
\begin{matrix}
\hatr^\gamma
\\
\hattheta
\end{matrix}
\right]
\end{equation}
with the partial derivatives
\begin{equation}
\partial_{\hatr} F_{\hatP} 
= 
\left[
\begin{matrix}
\gamma \hatr^{\gamma -1}
\\
0
\end{matrix}
\right],
\qquad
\partial_{\hattheta} F_{\hatP} 
= 
\left[
\begin{matrix}
0
\\
1
\end{matrix}
\right]
\end{equation}
The induced metric has elements 
$(G_{\hatP})_{ij}  = \partial_i F_{\hatP} \cdot G_P \cdot \partial_j F_{\hatP}$, 
where we use the polar metric inner product,
\begin{equation}
G_{\hatP} 
= 
\left[ 
\begin{matrix}
\gamma^2 \hatr^{2(\gamma -1)} & 0
\\
0 & r^2
\end{matrix}
\right]
=
\left[ 
\begin{matrix}
\gamma^2 \hatr^{2(\gamma -1)} & 0
\\
0 & \hatr^{2\gamma}
\end{matrix}
\right]
=
 \hatr^{2(\gamma -1)} \left[ 
\begin{matrix}
\gamma^2 & 0
\\
0 & \hatr^{2}
\end{matrix}
\right]
\end{equation}
with inverse
\begin{equation}
G_{\hatP} ^{-1}
= 
 \hatr^{-2(\gamma -1)} \left[ 
\begin{matrix}
\gamma^{-2} & 0
\\
0 & \hatr^{-2}
\end{matrix}
\right]
\end{equation}
and determinant
\begin{equation}
|G_{\hatP}| = \gamma^2 \hatr^{(4\gamma - 2)} , 
\qquad 
|G_{\hatP}|^{1/2}  =   \gamma \hatr^{2(\gamma-1)} \hatr
\end{equation}

\paragraph{Bilinear Form.}
The bilinear form associated with the Laplacian transforms as follows
\begin{align}\nonumber
&\int_{B_P} (\nabla v)_P \cdot G_P  \cdot (\nabla_P w)_P |G_P|^{1/2} dr d\theta
\\ 
&\qquad = \int_{B_{\hatP}} (\nabla v)_{\hatP} \cdot G_{\hatP}  \cdot (\nabla w)_{\hatP}  |G_{\hatP}|^{1/2} d\hatr d \hattheta
\\
&\qquad =
\int_{B_{\hatP}} \nabla_{\hatP}  v_{\hatP} \cdot G^{-1}_{\hatP}  \cdot \nabla_{\hatP}  w_{\hatP} 
 |G_{\hatP}|^{1/2} d\hatr d \hattheta
\\
&\qquad =
\int_{B_{\hatP}} 
\Big( \gamma^{-1} \partial_{\hatr} v_{\hatP} \partial_{\hatr} w_{\hatP} 
+ \gamma \hatr^{-2} \partial_{\hattheta} v_{\hatP} \partial_{\hattheta} w_{\hatP} \Big) 
\hatr d\hatr d \hattheta
\end{align} 
Thus we obtain a form with a $\gamma$ dependent scaling of the radial and angular derivatives and 
in particular we note that no radial weight is present.
Transforming back to Cartesian coordinates in the reference domain we have the identity 
\begin{align}\nonumber
&\int_{B_{\hatP}} 
\Big( \gamma^{-1} \partial_{\hatr} v_{\hatP} \partial_{\hatr} w_{\hatP} 
+ \gamma \hatr^{-2} \partial_{\hattheta} v_{\hatP} \partial_{\hattheta} w_{\hatP} \Big) 
\hatr d\hatr d \hattheta
%\\&\qquad
=\int_{\widehat{B}} 
\widehat{\nabla} \hatv\cdot S_{\hattheta}^{T} D_\gamma S_{\hattheta} \cdot \widehat{\nabla} \hatv \;
 d\hatx d\haty
\end{align}
where 
\begin{equation}
D_\gamma 
= \left(
\begin{matrix}
\gamma^{-1} &  0
\\
0 & \gamma 
\end{matrix}
\right),
\qquad 
S_{\hattheta} 
= \left(
\begin{matrix}
 \cos \hattheta &  \sin \hattheta
\\
-\sin \hattheta & \cos \hattheta
\end{matrix}
\right)
\end{equation}
and we used the identities
\begin{equation}
\left( 
\begin{matrix}
\partial_{\hatr} v
\\
{\hatr}^{-1} \partial_{\hattheta} \hatv 
\end{matrix}
\right)
=
S_{\hattheta}
\left( 
\begin{matrix}
\partial_{\hatx} \hatv
\\
 \partial_{\haty}  \hatv 
\end{matrix}
\right), \qquad \hatr d\hatr d \hattheta = d \hatx d \haty
\end{equation}

\paragraph{Consistency Term.}
Correspondingly the consistency term transforms as follows
\begin{align}
\int_{\partial B_P} ((n)_P \cdot G_P \cdot (\nabla v)_P) w_P \, ds_P
&=
\int_{\partial B_{\hatP}} ((n)_{\hatP} \cdot G_{\hatP} \cdot (\nabla v)_{\hatP}) w_{\hatP} \, (ds)_{\hatP}
\\&=
\int_{\partial B_{\hatP}} (n_{\hatP} \cdot G_{\hatP}^{-1} \cdot \nabla_{\hatP} v_{\hatP}) w_{\hatP} \gamma \hatr^{2(\gamma - 1)} \hatr  \, ds_{\hatP}
\\&=
\int_{\partial B_{\hatP}} \left(n_{\hatP} \cdot
\begin{bmatrix}
1 & 0 \\ 0 & \hatr^{-1}
\end{bmatrix}
D_\gamma 
\begin{bmatrix}
1 & 0 \\ 0 & \hatr^{-1}
\end{bmatrix}
\cdot \nabla_{\hatP} v_{\hatP} \right) w_{\hatP} \hatr  \, ds_{\hatP}
\end{align}
where $ds_{\hatP}$ is the standard line measure on $\IR^2$.
Here we used that
\begin{align}
(n)_{\hatP} = G_{\hatP}^{-1} n_{\hatP} (n_{\hatP} \cdot G_{\hatP}^{-1} \cdot n_{\hatP})^{-1/2} \,,
\qquad
(ds)_{\hatP} = (t_{\hatP} \cdot G_{\hatP} \cdot t_{\hatP})^{1/2} ds_{\hatP}
\end{align}
and also utilized the simplification
\begin{equation}
\left(
\frac{t_{\hatP} \cdot G_{\hatP} \cdot t_{\hatP}}{n_{\hatP} \cdot G_{\hatP}^{-1} \cdot n_{\hatP}}
\right)^{1/2}
=
\gamma \hatr^{2(\gamma - 1)} \hatr
\end{equation}
Transforming back to Cartesian coordinates in the reference domain we have the identity
\begin{align}
\int_{\partial B_{\hatP}} \left(n_{\hatP} \cdot
\begin{bmatrix}
1 & 0 \\ 0 & \hatr^{-1}
\end{bmatrix}
D_\gamma 
\begin{bmatrix}
1 & 0 \\ 0 & \hatr^{-1}
\end{bmatrix}
\cdot \nabla_{\hatP} v_{\hatP} \right) w_{\hatP} \hatr  \, ds_{\hatP}
&=
\int_{\partial \widehat{B}} \left(\hatn \cdot
S_{\hattheta}^T D_\gamma S_{\hattheta}
\cdot \hatnabla \hatv \right) \hatw  \, d\hats
\end{align}
where we used the identities
\begin{equation}
\begin{pmatrix}
n_{\hatr} \\ \hatr^{-1} n_{\hattheta}
\end{pmatrix}
=
S_{\hattheta}
\begin{pmatrix}
n_{\hatx} \\ n_{\haty}
\end{pmatrix}
=
S_{\hattheta}
\hatn
, \qquad \hatr ds_{\hatP} = d \hats
\end{equation}

%\item
%The line measure $ds$ for physical Cartesian coordinates has the following
%relation to the line measure $d\hats$ in reference Cartesian coordinates.
%\begin{equation}
%ds^2 =
%\gamma \hatr^{2(\gamma-1)} (\hatn^T S_{\hattheta}^{T} D_\gamma S_{\hattheta} \hatn) \, d\hats^2
%\end{equation}

\end{document}